\documentclass[reqno,12pt,a4paper]{amsart}

\numberwithin{equation}{section}
\allowdisplaybreaks 

\usepackage[normalem]{ulem}

\usepackage{amssymb,eucal,mathrsfs,setspace,bm}

\usepackage[all]{xy}

\usepackage[noadjust]{cite}
\usepackage[inner=25mm, outer=25mm, top=25mm, bottom=25mm, head=10mm, foot=10mm]{geometry}

\usepackage{array}
\newcolumntype{C}{>{$}c<{$}} 

\usepackage[largesc,theoremfont]{newtxtext}      
\usepackage[libertine,cmbraces]{newtxmath} 
\begingroup
  \makeatletter
  \@for\theoremstyle:=definition,remark,plain\do{%
    \expandafter\g@addto@macro\csname th@\theoremstyle\endcsname{%
      \addtolength\thm@preskip{.5\baselineskip plus .2\baselineskip minus .2\baselineskip}
      \addtolength\thm@postskip{.5\baselineskip plus .2\baselineskip minus .2\baselineskip}
    }%
  }
\endgroup

\usepackage{enumitem}
\setitemize{leftmargin=*}   
\setenumerate{leftmargin=*, 
  label=\textup{(\alph*)},  
	ref=\textup{\alph*}}      

\usepackage{mathtools} 
\usepackage{graphicx}

\usepackage[colorlinks=true,citecolor=red,linkcolor=blue]{hyperref} 

\usepackage[capitalise,noabbrev]{cleveref} 

\renewcommand{\cong}{\simeq}

\DeclareMathOperator{\id}{id}
\DeclareMathOperator{\im}{im} 

\DeclarePairedDelimiterX{\comm}[2]{\lbrack}{\rbrack}{#1 , #2}  
\DeclarePairedDelimiterX{\acomm}[2]{\lbrace}{\rbrace}{#1 , #2} 
\DeclarePairedDelimiterX{\inner}[2]{\langle}{\rangle}{#1 , #2} 
\DeclarePairedDelimiterX{\super}[2]{\lparen}{\rparen}{#1 \delimsize\vert \mathopen{} #2} 

\newcommand{\ra}{\rightarrow}

\newcommand{\fld}[1]{\mathbb{#1}}    
\newcommand{\alg}[1]{\mathfrak{#1}}  

\newcommand{\ZZ}{\fld{Z}}
\newcommand{\QQ}{\fld{Q}}

\newcommand{\CC}{\fld{C}}

\newcommand{\ah}{\alg{h}}

\theoremstyle{plain}
\newtheorem{theorem}{Theorem}
\newtheorem{corollary}[theorem]{Corollary}
\newtheorem{lemma}[theorem]{Lemma}
\newtheorem{proposition}[theorem]{Proposition}
\newtheorem{remark}[theorem]{Remark}

\usepackage[all]{xy}

\newcommand{\lam}{\lambda}
\newcommand{\Lam}{\Lambda}

\newcommand{\ch}{\mathrm{ch}}

\newcommand{\hh}{\widehat{\mathfrak{h}}}

\newcommand{\Znn}{\mathbb{Z}_{\geq0}}
\newcommand{\Zp}{\mathbb{Z}_{>0}}
\newcommand{\Zn}{\mathbb{Z}_{<0}}
\newcommand{\Znp}{\mathbb{Z}_{\leq0}}
\newcommand{\cspan}{\mathrm{span}_{\CC}}

\newcommand{\cB}{\mathcal{B}}

\newcommand{\cG}{\mathcal{G}}

\newcommand{\geqc}[1]{ \underset{{\tiny #1}}{\geq}}
\newcommand{\leqc}[1]{ \underset{{\tiny #1}}{\leq}}
\newcommand{\asltwo}{\widehat{\mathfrak{sl}}_2}

\newcommand{\biadd}{\Delta}

\newcommand{\cwb}{D}
\newcommand{\chb}{C}
\newcommand{\chm}{S}
\newcommand{\cwm}{E}

\newcommand{\qbinom}{\genfrac{[}{]}{0pt}{}}
\newcommand{\fd}[2]{\mathrm{M}_{#1}\left(#2\right)}
\newcommand{\fda}[2]{F(#1,#2)}
\newcommand{\ef}[2]{\mathrm{EF}_{#1}\left(#2\right)}
\newcommand{\rf}[2]{\mathrm{RF}_{#1}\left(#2\right)}
\newcommand{\M}[1]{\mathrm{M}\left(#1\right)}
\newcommand{\Mcirc}[1]{\mathrm{M}^\circ\left(#1\right)}

\newcommand{\coker}[1]{\mathrm{Coker}\left(#1\right)}
\renewcommand{\ker}[1]{\mathrm{Ker}\left(#1\right)}

\newcommand{\alp}{\alpha}

\makeatletter
\renewcommand\author@andify{%
  \nxandlist {\unskip ,\penalty-1 \space\ignorespaces}%
    {\unskip {} \@@and~}%
    {\unskip \penalty-2 \space \@@and~}%
}
\makeatother

\begin{document}

\title[]{Rogers-Ramanujan exact sequences and 
free representations over free generalized vertex algebras}

\author[K~Kawasetsu]{Kazuya Kawasetsu}
\address[Kazuya Kawasetsu]{
Priority Organization for Innovation and Excellence\\
Kumamoto University\\
Kumamoto, Japan, 860-8555.
}
\email{kawasetsu@kumamoto-u.ac.jp}

\begin{abstract}
The Rogers-Ramanujan recursions are studied from the viewpoint of free representations over free (generalized) vertex algebras.
Specifically, we construct short exact sequences among the free representations over free generalized vertex algebras 
 which lift the recursions.
They  naturally generalize  exact sequences 
introduced by S.~Capparelli et al.
 We also show that an analogue of the Rogers-Ramanujan recursions is realized as 
an exact sequence among 
finite-dimensional free representations over a certain finite-dimensional vertex algebra.

As an application of exact sequences, we reveal a relation between free generalized vertex algebras and $\widehat{\mathfrak{sl}}_2$ spaces of coinvariants $L_{1,0}^{N,\infty}(\mathfrak n)$ introduced by B.~Feigin et al.
Moreover, it is shown that generalized $q$-Fibonacci recursions may be realized as exact sequences obtained by applying several functors to exact sequences among free representations over free generalized vertex algebras.
Finally, we consider a character decomposition formula of the basic $\widehat{\mathrm{sl}}_2$-module $L_{1,0}$  by M.~Bershtein et al.~related to the Urod vertex operator algebras, from the viewpoint of free representations and exact sequences.
\end{abstract}

\maketitle

\onehalfspacing

\section{Introduction} \label{sec:intro}

Let $g$ be a complex number and consider the series
 $$
 F_g(z,q)=\sum_{r=0}^\infty \frac{z^r q^{gr^2/2}}{(q)_r}=1+\frac{zq^{g/2}}{1-q}+\frac{z^2g^{2g}}{(1-q)(1-q^2)}+\cdots,
 $$
where $1/(1-q^i)$ is expanded into the sum $1+q^i+q^{2i}+\cdots$.
Note that the sums $F_2(1,q)$ and $F_2(q,q)$ are the sum sides of the celebrated Rogers-Ramanujan (RR) identities.
The series  satisfies the recursion relation
\begin{align}\label{genrrbasic}
&F_g(z,q)=F_g(zq,q)+zq^{g/2}F_g(zq^g,q).
\end{align}
When $g=2$, it is the {\em Rogers-Ramanujan (RR) recursion},
which plays crucial roles in the proof of the RR identities in \cite{RR} (see e.g.~\cite{A}).

In this paper, we study \eqref{genrrbasic} from the viewpoint of 
 {\em free representations} over {\em free (generalized) vertex algebras}.
 Specifically, we construct exact sequences among the free representations over free generalized vertex algebras which
 lift \eqref{genrrbasic}. 
 They may be seen as a natural generalization of exact sequences introduced  in \cite{CLM03} and extended in \cite{CalLM14}, \cite{PSW} and others.

Let us explain.
For simplicity, let us consider the cases of $g\in \ZZ$, which correspond to usual vertex algebras.
The general cases of $g\in \CC$ correspond to the theory of {\em generalized vertex algebras}, which will be mentioned later.

A {\em vertex algebra}  is a vector space $V$ equipped with a system of multiplications
$((n):V\times V\ra V)_{n\in\ZZ}$ satisfying certain axioms \cite{B}. 
Vertex algebras  
 appear in 2d (and higher-dimensional) conformal field theory in physics
and are becoming more and more ubiquitous in both mathematics and physics.
An important axiom of vertex algebras relevant to this paper is
 the {\em field axiom} 
 that for any $u,v\in V$, if $n$ is big enough
then $u(n)v=0$.
It implies that there are no free objects in the category of vertex algebras equipped with a fixed number of generators.
The field axiom rather forces the following notion of {\em free vertex algebras} 
 with respect to not only the number of generators but also data specifying when $(n)$-th products among generators become zero, as mentioned by R.~Borcherds \cite{B} and constructed by M.~Roitman  \cite{R}. 

Let $g$ be an integer.
A (singly generated)  free vertex algebra $F(g)$ 
 is a vertex algebra freely generated by a formal element $b$ satisfying $b(-g+n)b=0$ for any $n\geq 0$.
More precisely, it is an initial object in the category of 
vertex algebras with a fixed generator $b$ satisfying the same condition.
The vertex algebra $F(g)$ is infinite-dimensional and have a 
PBW like monomial basis satisfying the so-called difference $g$ condition.
The basis implies that $F(g)$ admits a bigrading $F(g)=\bigoplus_{r,d}F(g)_{r,d}$
such that a monomial $b(-n_r)\cdots b(-n_1)\bm1$ lies in $F(g)_{r,n_1+\cdots+n_r}$.
It then follows that the generating function 
$\chi_g(z,q)=\sum_{r,d}\dim(F(g)_{r,d})z^rq^d$ (the {\em character}
of $F(g)$) coincides with $F_g(zq^{1-g/2},q)$, which is a normalization of $F_g(z,q)$.

 A {\em module} over a vertex algebra $V$ is roughly speaking a vector space $M$ equipped with a system of actions 
 $((n):V\times M\ra M)_{n\in\ZZ}$ which satisfies the field axiom that for all $u\in V$ and $v\in M$, if $n\gg0$ then $u(n)v=0$.
 Or, we may say that a $V$-module is a module over a universal enveloping associative algebra $U(V)$ of $V$ such that the field axiom holds.
Note that $U(V)$ is defined as a topological vector space because an axiom of vertex algebras is an
identity among infinite sums of products of left multiplications.
The field axiom guarantees that the infinite sums in the axiom actually act as  finite sums on  $V$ and $M$, which are meant to be usual vector spaces.

The situation makes us to consider the following notion of {\em free representations}.
 Let us assume for simplicity that $V$ is generated by a single generator $b$.
 Let $m$ be an integer.
 A  free representation  of order $-m$  is  an initial object in the category of $V$-modules 
with a fixed single generator $v_m$ satisfying $b(-m+n)v_m=0$ for all $n\geq0$.
Note that our free representations  are usually not isomorphic to direct sums of $V$.
However, ones over
a free vertex algebra $F(g)$ turn out to have essentially the same characters as that of $F(g)$.
The free representation $\M m$ of order $-m$ over  $V=F(g)$ is non-zero and admits a PBW like basis. 
 Then the {\em character} $\chi_{g;m}(z,q)$ of $\M m$ 
 has the form
$\chi_{g;m}(z,q)=F_g(zq^{1-g/2+m},q)$, which is a normalization of $F_g(z,q)$ again. 

 The initiality of free representations implies that there is a natural
 exact sequence of the form $\M{m+g}\ra \M m\ra \M{m+1}\ra 0$.
If  $V=F(g)$, it follows that the first map is injective, namely,
we have a short exact sequence of the form
 \begin{equation}\label{ourintro-}
 0\ra \M {m+g}\ra \M m\ra \M {m+1} \ra 0
 \end{equation}
 of $F(g)$-modules.
It implies an identity among three normalized $F_g(z,q)$ 
 and it corresponds to \eqref{genrrbasic} after further normalization,
showing that \eqref{ourintro-} is a lift of a RR recursion.

%
%

Our point of view generalizes to recursions among finite analogues of $F_g(z,q)$.
Let $p$  be a positive integer and $\ell$ a non-negative integer.
Consider the polynomial 
$$
 F_{p,\ell}(z,q)=\sum_{r\geq 0} z^rq^{pr^2/2}\qbinom{\ell-(p-1)r}r_q,
 $$
where the symbol $\qbinom n r_q$ denotes the $q$-binomial coefficient.
The polynomial $F_{2,\ell}(z,q)$ coincides with the Carlitz $q$-Fibonacci polynomial \cite{Ca}.
It is related to finite analogues of the RR identities introduced in \cite{S} (see e.g., \cite{A04,Ci}).
The polynomial $F_{p,\ell}(z,q)$ satisfies recursion relations
\begin{align}
&F_{p,\ell+p}(z,q)=F_{p,\ell+p-1}(z,q)+zq^{\ell+p/2}F_{p,\ell}(z,q), \label{fibbasic}\\
&F_{p,\ell+p}(z,q)=F_{p,\ell+p-1}(zq,q)+zq^{p/2}F_{p,\ell}(zq^p,q).\label{finrrbasic}
\end{align}
The former with $p=2$ is the $q$-Fibonacci  recursion and the latter is
an analogue of \eqref{genrrbasic}.
The polynomial $F_{2,\ell}(z,q)$ also appears in a one-dimensional problem in Baxter's Hard Hexagon model and 
recursions \eqref{fibbasic}--\eqref{finrrbasic} naturally arise (cf.~\cite{OSS, Ci}).

We realize \eqref{finrrbasic} as an exact sequence among free representations of order $-m\in\ZZ$ over a certain finite-dimensional vertex algebra $\fda kp$ $(k\in\ZZ)$:
\begin{equation}\label{finrrintro}
0\ra \fd k{m+g}\ra \fd km\ra \fd k{m+1},
\end{equation}
where $\ell$ in \eqref{finrrbasic} is related to $k$ and $m$.
Here $\fd km$ is a free module of order $-m$ over $F(k,p)$.

By considering $\fd km$ as an $F(p)$-module via a surjection $F(p)\twoheadrightarrow \fda kp$, we may derive \eqref{finrrintro} from \eqref{ourintro-} with $g=p$ by 
taking  quotients.
We then observe that \eqref{ourintro-} with {\em non-integral} $g=1/p$ induces an exact sequence for \eqref{fibbasic}.
(We explain it more in \Cref{subsecintrofib}.)
This is one of the reasons why we consider any complex number $g$  in the beginning of the introduction.
We have the notion of a {\em free GVA} $F(g)$ with $g\in\CC$ \cite{Kaw15}, so that we can consider $F(1/p)$.
Here a {\em generalized vertex algebra} (GVA) is roughly a vector space which admits $(n)$-th products with $n\in\CC$ \cite{DL}.
The exact sequence \eqref{ourintro-}  generalizes to any $g\in \CC$ and $m\in\CC$.

\medskip
We now explain applications of our exact sequences. There are three kinds of applications.

\subsection{Applications to the theory of principal subspaces and $\asltwo$ spaces of coinvariants}
In \cite{R}, an embedding of $F(g)$ into a 
{\em lattice vertex operator algebra} (VOA) is constructed. 
We show that the free modules $\M m$ are also embedded into modules over lattice VOAs and the images coincide with the
 {\em lattice principal subspaces}  $W_L(m\beta^\circ)$ \cite{MP},
which generalize {\em Feigin-Stoyanovsky (FS) principal subspaces} of Kac-Moody Lie algebra modules \cite{SF}.
We then obtain exact sequences 
\begin{equation}\label{exactpsintro}
0\ra W_L((m+g)\beta^\circ)\hookrightarrow W_L(m\beta^\circ) \xrightarrow{e^{\beta^\circ}(-g^{-1}m-1)} W_L((m+1)\beta^\circ)\ra 0
\end{equation}
 (see \Cref{subsecRRprin} for more details). 
They are equivalent to the exact sequence in \cite{CLM03} ($g=2$) and rank one cases in \cite{PSW} ($g\in\Zp$)  if we compose them with certain simple current operators.
We remark that in \cite{CLM03}, S.~Capparelli, A.~Milas and J.~Lepowsky brought the
theory of FS principal subspaces into the theory of vertex algebras and lifted 
the RR recursion \eqref{genrrbasic} with $g=2$ to the exact sequence equivalent to \eqref{exactpsintro}.
See also \cite{CalLM08,CalLM14,PSW} and others.
Our results naturally generalize the exact sequence in \cite{CLM03}.

We then apply the exact sequences to construct lattice principal subspaces as certain invariance subspaces (\Cref{propinv}).
Moreover, we relate  $\asltwo$ spaces of coinvariants $L_{1,0}^{N,\infty}(\mathfrak n)$ studied extensively in \cite{FKLMM} and others with 
certain modules over a free GVA $F(1/2)$.
It implies a proof of the fermionic character formula of $L_{1,0}^{N,\infty}(\mathfrak n)$. 
See the end of \Cref{subseccoinv} for more details.

\subsection{Applications to Fibonacci recursions}\label{subsecintrofib}
Recall generalized Fibonacci recursion \eqref{fibbasic}.
We lift it to an exact sequence of the form
\begin{equation}\label{finrr2intro}
0\ra \fd {k-p}m\ra \fd km\ra \fd {k-1}m,
\end{equation}
where the modules are considered as  $F(p)$-modules (\Cref{propfib}).
Unlike \eqref{finrrintro}, it does not appear to be a consequence of \eqref{ourintro-} with $g=p$.
As we have already mentioned, we actually show that it can be induced from  \eqref{ourintro-} with $g=1/p$  by applying several functors
(\Cref{subsecswitch}).
To show this, we 
construct isomorphisms between the restricted dual 
modules of $\fd km$ and certain $F(p)$-modules  coming from free $F(1/p)$-modules (\Cref{corswitch}).
The isomorphisms  lift a certain combinatorial identity involving $q$-binomial coefficients of the form 
$\qbinom n r_q$ and $\qbinom n{pr}_q$ (\Cref{corchar}), e.g.,
\begin{equation}\label{eqn2rr}
z^{n}q^{n^2}\sum_{r=0}^{n}q^{r^2-2nr}
\begin{bmatrix}
2n-r\\r
\end{bmatrix}_q
z^{-r}
=
\sum_{r=0}^{n}q^{r^2}\begin{bmatrix}
n+r\\2r\end{bmatrix}_q z^{r}.
\end{equation}

We illustrate the relations among our exact sequences in \Cref{fig1}.

\begin{figure}
$
\xymatrix{
\mbox{Freeness of $\M m$}\ar@{-->}[d]&
\mbox{RR for $W_L(m\beta^\circ)$ \eqref{exactpsintro}}&
\mbox{Freeness of $\fd km$}\ar@{-->}[d]
\\
\mbox{RR \eqref{ourintro-}}\ar[d]\ar[r]\ar[ur]
&(g=p) \ar[r]_{quotient}
&\mbox{RR for $\fd km$ \eqref{finrrintro}}
\\
(g=1/p)\ar[r]^{quotient}_{summand}& \mbox{(RR for RF \eqref{efrrsummand})}\ar[r]_{rest.~dual\quad}&
\mbox{Fib for $\fd km$ \eqref{finrr2intro}}
}
$ 
\caption{Relations between exact sequences}\label{fig1}
\end{figure}
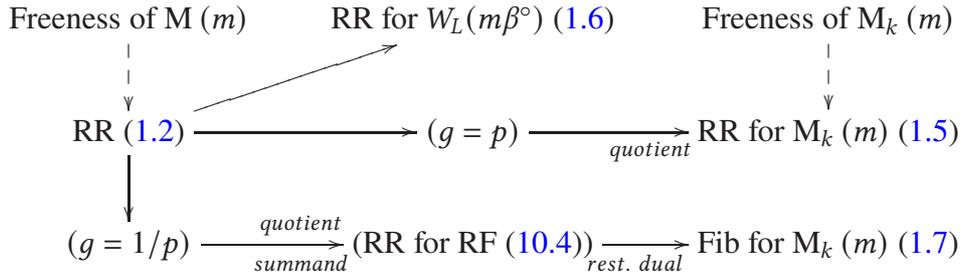

\subsection{Applications to the BFL character decomposition}
Finally, we consider a character decomposition formula \eqref{eqnmainchar}  of lattice VOA modules appearing in  \cite{BFL} and recent preprint \cite{Ke2}.
We relate the formula to the characters of free modules over $F(p)$ and $F(1/p)$.
    Note  that if $p=2$, the formula \eqref{eqnmainchar} with $z=1$ is a consequence of the decomposition 
    in terms of the Urod VOAs \cite{BFL}.

\subsection{Features}
Our construction suggests that RR recursion \eqref{genrrbasic}  arises the most naturally in the framework of free representations over free vertex algebras.
We may also say that the RR recursion is
  intrinsic in the definition of  vertex algebras.
  Our consideration of GVAs allows us to consider Fibonacci type recursions and $\asltwo$ spaces of coinvariants from our viewpoint.

   We note that free vertex algebras related to quivers recently appear in the study of Donaldson-Thomas invariants \cite{DM}.
We also note that  the fermionic  $q$-character formula of $L_{1,0}^{0,\infty}(\mathfrak n)$  coincides with that of a Virasoro minimal model of central charge $-3/5$ \cite{KKMM}.
The bosonic character formula of $L_{1,0}^{0,\infty}(\mathfrak n)$
is obtained in \cite{JMM}.
Moreover, the algebra $W_{A_1^\circ}\cap V_{A_1}\cong F(1/2)^{(\ZZ)}$, which is isomorphic to the restricted dual of $L_{1,0}^{0,\infty}(\mathfrak n)$, coincides with the commutant of $W_{A_1}$ inside $V_{A_1}$ studied extensively in \cite{Kaw23a}.


%
 
%

\subsection{Organization}
The paper is divided into four parts.
The first part consists of \Cref{secgva}--\Cref{secfin}.
In \Cref{secgva}, we recall the notion of GVAs and their modules.
The free GVAs are briefly explained in \Cref{secfree}.
In \Cref{sec:freemod}, we study free representations over  free GVAs and show
that \eqref{ourintro-} is exact.
In \Cref{secfin}, we introduce the finite-dimensional vertex algebra $\fda kp$ 
and \eqref{finrrintro} is shown to be exact.
 
The second part of the paper consists of \Cref{secpreapp}--\Cref{secinvcoinv}.
In \Cref{secpreapp}, we  recall the notion of lattice GVAs and restricted dual modules.
In \Cref{secps}, we explain the relationship between our RR exact sequences and ones among lattice principal subspaces.
We then apply the exact sequence to construct lattice principal subspaces by using certain invariance construction in  \Cref{secinvcoinv}. The relation to $\asltwo$ spaces of coinvariants is also established.

\Cref{secfib}--\Cref{secswitch} are the third part of the paper.
In \Cref{secfib}, we show that \eqref{finrr2intro} is exact. 
A goal of this part is to relate it with \eqref{ourintro-}.
The RF-type finite-dimensional modules are introduced in \Cref{secrf}.
In \Cref{secswitch}, we prove that 
the restricted dual modules of $\fd km$ are isomorphic to RF-type modules, which shows the relation between \eqref{finrr2intro} and \eqref{ourintro-}.

Finally, 
the Urod type character decomposition is discussed from the viewpoint of exact sequences in  \Cref{seccomp}.



\subsection*{Notations.}
All vector spaces in this paper are over the field of complex numbers $\CC$.
The set of integers, rational numbers, non-zero complex numbers are denoted by $\ZZ$,
 $\QQ$ and $\CC^\times$, respectively.
 The set of non-negative (resp., positive, non-positive, negative, non-zero) integers is denoted by $\Znn$ (resp., $\Zp,\Znp,\Zn, \ZZ^\times$).
 For $n,m,k\in\CC$, we write $n\geqc k m$ if $n-m\in k+\Znn$.
 We set $(q)_0=1$ and $(q)_r=(1-q)(1-q^2)\cdots (1-q^r)$.
Let $V$ be a vector space.
We set 
 $V[[z]]=\{\sum_{n\in\Znn} a_nz^n\,|\,a_n\in V\ (n\in\Znn)\}$ and
 $V\{z\}=\{\sum_{n\in\CC} a_nz^n\,|\,a_n\in V\ (n\in\CC)\}$.
 For any $c+\ZZ\in \CC/\ZZ$, we write 
 $V[[z]]z^{c+\ZZ}=\{\sum_{n=N}^\infty a_nz^{n+c}\,|\,N\in\ZZ,a_n\in V\ (n\geq N)\}$.
Let $W$ be a vector subspace of $V$.
The class of $v\in V$ in the quotient $V/W$ is often written as $[v]$ and so is the class of subsets of $V$.

\subsection*{Acknowledgements}

The author thanks Prof.~Atsushi Matsuo for fruitful discussions and advice on the original manuscripts.
He also thanks Profs.~Ching Hung Lam
and Hiroshi Yamauchi for helpful comments and discussions.
This research is partially supported by
MEXT Japan ``Leading Initiative for Excellent Young Researchers (LEADER)'',
JSPS Kakenhi Grant numbers 19KK0065, 21K13775 and 21H04993.

%

\section{Generalized vertex algebras}\label{secgva}
In this section, we briefly recall the definition of generalized vertex algebras (GVAs) and their modules \cite{DL}.
We follow the conventions in \cite{BK}.

\subsection{Definitions of generalized vertex algebras}\label{subsecgva}

Let $A$ be an abelian group equipped with a bimultiplicative function $\eta:A\times A\ra \CC^\times, (\gamma,\delta)\mapsto \eta_{\gamma,\delta}$, called a {\em phase}.
 We have a biadditive map $\biadd=\biadd_\eta:A\times A\ra \CC/\ZZ$ associated to $\eta$ uniquely determined by the relation $\eta_{\gamma,\delta}\eta_{\delta,\gamma}=e^{-2\pi\sqrt{-1}\biadd(\gamma,\delta)}$.
Let $V=\bigoplus_{\gamma\in A}V^\gamma$ 
 be an $A$-graded vector space.
 A homogeneous (parafermionic) {\em field} of degree $\gamma\in A$ on $V$ is a formal series
 $u(z)\in \mathrm{End}(V)\{z\}$
such that for any homogeneous element $v\in V^\delta$ ($\delta\in A$), we have $u(z)v\in V^{\gamma+\delta}[[z]]z^{-\biadd(\gamma,\delta)}$.
Consider a linear map $Y(\cdot,z):V\ra \mathrm{End}(V)\{z\}$.
Let $\bm 1\in V^0$ be a distinguished element called the vacuum vector and $\partial:V\ra V$ an $A$-grading preserving linear operator.
The tuple $(V=\bigoplus_{\gamma\in A}V^\gamma,Y,\bm1,\partial,\eta)$ is called an {\em $A$-graded GVA} with the phase $\eta$
if, for any $u\in V^\gamma$ and $v \in V^\delta$ ($\gamma,\delta\in A$), we have
\begin{enumerate}
\item (vacuum axiom) $Y(\bm1,z)=\id_V$, $Y(u,z)\bm1\in u+V[[z]]z$, $\partial \bm1=0$,
\item (translation covariance) $[\partial, Y(u,z)]=\partial_zY(u,z)$,
\item (field axiom) $Y(u,z)$ is a homogeneous field of degree $\gamma$,
\item (locality) there is $N\in\biadd(\gamma,\delta)$ such that
\begin{equation}\label{gvaloc}
(z-w)^NY(u,z)Y(v,w)-\eta_{\gamma,\delta}e^{\pi\sqrt{-1}N}(w-z)^NY(v,w)Y(u,z)=0.
\end{equation}
\end{enumerate}
 Here powers of the form $(x-y)^n$ are always expanded in the positive powers of the second variable $y$.
We often just write $V$ or $V=\bigoplus_{\gamma\in A}V^\gamma$.
We set $u(n)=\mathrm{Res}_z z^nY(u,z)$, where  $\mathrm{Res}_z$ denotes the formal residue in $z$, so that  $Y(u,z)=\sum_{n\in\CC}u(n)z^{-n-1}$.
If $A=\ZZ/2\ZZ$ and $\eta_{\gamma,\delta}=(-1)^{\gamma\delta}$,
then $\biadd$ is trivial and $V$ is nothing but a vertex superalgebra.
If  $\eta$ is trivial,  $V$ is a usual vertex algebra (with the $A$-grading).
In particular, if $A=0$, then $V$ is a vertex algebra.
In the following, we just call vertex superalgebras  vertex algebras.

We obviously have the notions of $A$-graded sub\,GVAs, ideals and quotients.

\subsection{Locality order}\label{subsecloc}
Consider the locality relation \eqref{gvaloc}.
If $Y(u,z)Y(v,w)\neq0$, the minimal number $N=N_0$ satisfying \eqref{gvaloc}  coincides with   the number $N_0$ defined by the relation
$$
u(N_0-1)v\neq0,\quad u(N_0+n)v=0 \quad (n\in\CC\setminus \Zn)
$$
(\cite[Corollary 3.6]{BK}) and called the locality order of the pair $(u,v)$.
If $u=v$, we call it the locality order of $u$.
We see that \eqref{gvaloc} holds if and only if $N=N_0,N_0+1,N_0+2,\ldots$.
We note that 
the locality order of $(u,\bm1)$ is $N_0=0$.

Let $\gamma$ be an element of $A$.
If $N=N_0$ is the locality order of a non-zero element $u\in V^\gamma$, then we see by \eqref{gvaloc} that $\eta_{\gamma,\gamma}=e^{-\pi\sqrt{-1}N}$.

If the locality order of $(u,v)$ is a non-positive
integer, then we may substitute $N=0$ in \eqref{gvaloc} and as a result, we have a commutation relation
\begin{equation}\label{commrel}
u(s)v(t)=\eta_{\gamma,\delta}v(t)u(s) \quad\mbox{for any}\ s,t\in\CC.
\end{equation}
In this case, we say that $u$ and $v$ {\em commute} with each other.
Graded subsets $U$ and $W$ of $V$ {\em commute} with each other if any homogeneous elements $u\in U$ and $v\in W$ commute.
A vertex algebra $V$ is {\em commutative} if it commutes with itself.
A commutative vertex algebra is equivalent to a commutative associative differential algebra \cite{B}.


\subsection{GV modules}
We recall the notion of modules over GVAs, which we call {\em GV modules}.
Let $V$ be a GVA with a phase $\eta$ and the associated biadditive map $\biadd$.
Let $B$ be a free additive $A$-set equipped with a biadditive map
$\biadd_1:A\times B\ra \CC/\ZZ$ 
extending $\biadd$.
Let $M=\bigoplus_{\lam\in B}M^\lam$ be a $B$-graded vector space equipped with a linear map $Y_M(\cdot,z):V\ra \mathrm{End}(M)\{z\}$.
The space $M$ is a $B$-graded {\em GV module} (or, {\em module}) over $V$
 if, for any $u\in V^\gamma$, $v\in V^\delta$ and $x \in M^\lam$ ($\gamma,\delta\in A,\lam\in B$), the following (1)--(3) hold:
(1) $Y_M(\bm1,z)=\id_M$,
(2) $Y_M(u,z)x\in M^{\gamma+\lam}[[z]]z^{-\biadd_1(\gamma,\lam)}$,
and (3) (Borcherds identity) for any $n\in\biadd(\gamma,\delta)$, we have
\begin{align*}
&(z-w)^nY_M(u,z)Y_M(v,w)x-\eta_{\gamma,\delta}e^{\pi\sqrt{-1}n}(w-z)^nY_M(v,w)Y_M(u,z)x\\
&\qquad=
\sum_{j\in\Znn} \frac1{j!}\partial_w^j\delta_{\biadd(\gamma,\lam)}(z,w)Y_M(u(n+j)v,w)x,
\end{align*}
where $\delta_{\biadd(\gamma,\lam)}(z,w)=\sum_{j\in \biadd(\gamma,\lam)}z^{-j-1}w^j$.
Here, we set $Y_M(u,z)=\sum_n u(n)z^{-n-1}$ with $u(n)\in\mathrm{End}(M)$.
We often just write $Y_M=Y$.
If $\eta$ and $\biadd_1$ are trivial, then $M$ is a usual vertex algebra module over the vertex algebra $V$.
If $A=0$, then all GV modules are vertex algebra modules.
Note that $V$ itself is naturally a $V$-module.

Let $M_1,M_2$ be $B$-graded GV modules over $V$, where $B$ is equipped with a common biadditive map $\Delta_1:A\times B\ra \CC/\ZZ$.
A $V$-module homomorphism is a $B$-grading preserving linear map $f:M_1\ra M_2$ such that $Y(f(u),z)f(v)=f(Y(u,z)v)$ for all $u\in V$ and $v\in M_1$.
We often encounter modules with different yet equivalent  gradings.
In this case, we may identify the gradings and then consider   homomorphisms.

\section{Free generalized vertex algebras}
\label{secfree}

In this section, we recall the theory of free generalized vertex algebras  \cite{Kaw15} (see also \cite{R}).

\subsection{Category $\cG(g)$}
Let $b$ be an indeterminate and $g$  a complex number.
Consider the phase $\eta:\ZZ\times\ZZ\ra \CC^\times$ 
defined by $\eta_{r,s}=e^{\pi\sqrt{-1}grs}$.
The associated biadditive map $\biadd=\biadd_\eta$ is given by $\biadd(r,s)=-grs+\ZZ$.
Consider the category $\cG(g)$ of $\ZZ$-graded GVAs
$V=\bigoplus_{\gamma\in \ZZ}V^\gamma$ with the phase $\eta$ generated by $b$ such that (1) $b$ is  homogeneous of degree 1, and (2) the locality relation
\begin{equation}\label{loce}
(z-w)^{-g}Y(b,z)Y(b,w)-(w-z)^{-g}Y(b,w)Y(b,z)=0
\end{equation}
holds. 
A morphism $f:V\ra W$ in $\cG(g)$ is a $\ZZ$-grading preserving linear map sending $b$ to $b$ and satisfying  (1) $f(\bm1)=\bm1$, (2) $f\circ \partial=\partial\circ f$,  and
(3) $f(Y(u,z)v)=Y(f(u),z)f(v)$ ($u,v\in V$).

We have an object $W_L$ in $\cG(g)$,  a (generalized) vacuum lattice
principal subspace of the lattice GVA $V_L$, where $L$ is a rank one  lattice generated by an element of square norm $g$ \cite{MP,Kaw15,R}. 
See \Cref{secps} for more details.

Before introducing the free GVAs, we give some technical assertions on objects in $\cG(g)$.
Let $V$ be an object of $\cG(g)$. 
The condition \eqref{loce} implies that the locality order of $b \in V$ is less than or equal to $-g$.
The relation \eqref{loce} implies the recursion relation that for any
$s,t\in\CC$,
\begin{align}\label{lemfund}
b(s)b(t)&=
\sum_{j=0}^\infty
(-1)^{j}\binom {-g} j b(t-g-j)b(s+g+j)
-\sum_{j=0}^\infty
(-1)^{j+1}\binom {-g} {j+1} b(s-1-j)b(t+1+j)
\end{align}
holds (cf.~{\rm \cite[Lemma 3.2]{Kaw15}}).
Since $V$ is generated by $b$, the space $V$ is spanned by
monomials in $b(n)$.
By applying the vacuum axiom and \eqref{lemfund} recursively to the monomials, we have the following proposition.
\begin{proposition} ({\rm cf.~\cite[Theorem 3.2]{Kaw15}})\label{propalg}
The following monomials span $V$:
\begin{equation} \label{basisfalg}
b(-n_r)\cdots b(-n_1)\bm1\quad (r\geq0, 
n_r\geqc g n_{r-1}\geqc g \cdots\geqc g n_1\geqc 0 1).
\end{equation}
Furthermore,
\begin{equation}\label{bzeroalg}
b(-n_r)\cdots b(-n_1)\bm1=0\quad \mbox{if }
n_1+\cdots+n_r\not\in \frac g2 r(r-1)+r+\Znn.
\end{equation}
\end{proposition}

Here $n\geqc k m$ means $n-m\in k+\Znn$ ($k\in\CC$).
For $V=W_L$, spanning set \eqref{basisfalg}
is a linear basis \cite{Geog,MP,Kaw15}.

Note that if $g$ is rational, then we may replace the $\ZZ$-grading in the definition of $\cG(g)$ to a smaller grading and the results of this paper  hold. 
For example, if $g$ is an integer then we may consider 
the $\ZZ/2\ZZ$-grading and if $g$ is even then the $0$-grading suffices.

\subsection{The free generalized vertex algebras}\label{subsecfreegva}
An initial object $F(g)$ in $\cG(g)$ is a {\em free GVA} freely generated by a single generator $b$. 
It is unique up to isomorphisms being an initial object and the existence is known in \cite{Kaw15,R}.
If $g\in\ZZ$, then $F(g)$ is a vertex (super)algebra, called a 
{\em free vertex algebra}.

By the initiality, we have a surjective morphism $F(g)\ra W_L$. Since \eqref{basisfalg} is a basis in $W_L$, the surjection is an isomorphism and vectors \eqref{basisfalg} form a combinatorial basis $\cB$ in $F(g)$.
In particular, the generator $b\in F(g)$ has  the locality order $N_0=-g$:
\begin{equation}\label{loce2}
b(-g-1)b\neq 0,\quad b(-g+n)b=0\quad \mbox{for all }n\in \CC\setminus\Zn.
\end{equation}

The {\em character} of $F(g)$ is the series 
$\chi_g(z,q)=\sum_{r\geq0,d\in \CC} \dim(F(g)_{r,d})z^rq^d$,
where the bigrading $F(g)=\bigoplus_{r,d} F(g)_{r,d}$ is defined  by
the condition $b(-n_r)\cdots b(-n_1)\bm1\in F(g)_{r,n_1+\cdots+n_r}$.
It follows by the basis that we have the character formula $\chi_g(z,q)=F_g(zq^{-g/2},q)$, which is sometimes called a
  fermionic formula.

We may also  define a normalized bigrading and the associated character as follows.
Let $C$ be a non-zero complex number and $D$ a complex number.
We denote by $F(g;C,D)$ the algebra $F(g)$ with the normalized bigrading $F(g;C,D)=\bigoplus_{r,d}F(g)_{r,d}$ such that
$b(-n_r)\cdots b(-n_1)\bm1\in F(g)_{Cr,n_1+\cdots+n_r-r+rD}$.
We have the  character formula
$
\chi_g^{(\chb,\cwb)}(z,q)
=F_g(z^\chb q^{D-g/2},q).
$
We often employ $C=1$ and $D=g/2$, where the character coincides with $F_g(z,q)$.

\subsection{A construction of free GVAs}\label{subsecconstfreegva}
For the sake of completeness, we give a quick construction of a free GVA $F(g)$ in this section,
which is different from ones in \cite{Kaw15,R}.
Let $A=\CC\langle b(n)\,|\,n\in\CC\rangle$ be a free associative  algebra freely generated by the symbols $b(n)$ with $n\in\CC$.
We have the translation operator $\partial:A\ra A$  defined by $\partial(1)=0$ and 
$[\partial,b(n)]=-nb(n-1)$, where $b(n):A\ra A$ is the left multiplication.
Let $I$ be the $\partial$-invariant left ideal of $A$ generated by the truncation relation $b(-n_r)\cdots b(-n_2)b(-n_1)=0$ with $n_1,\ldots,n_r\in\CC$ such that
$n_1+n_2+\cdots+n_r\not\in g r(r-1)/2+r+\Znn.$
Consider the quotient left $A$-module $A/I$.
The operator $b(z)=\sum_{n\in\CC}b(n)z^{-n-1}\in \mathrm{End}(A/I)\{z\}$ is a translation covariant (parafermionic) field on $A/I$.
Let $J\subset A/I$ be the $\partial$-invariant submodule generated by the coefficients of the locality relation
\begin{equation*}
(z-w)^{-g}b(z)b(w)v-(w-z)^{-g}b(w)b(z)v=0\quad \mbox{with}\quad v\in A/I.
\end{equation*}
Note that all coefficients in the above relation are finite sums 
since $b(z)$ is a field.
Now, the quotient space $F(g)=(A/I)/J$ is the underlying vector space of our free GVA.
The vacuum vector $\bm1$ of $F(g)$ is the image of $1$ and we have the induced translation operator $\partial:F(g)\ra F(g)$.
The image of $b(-1)\in A$ is denoted by $b\in F(g)$.
We readily see that the induced operator $b(z)\in \mathrm{End}(F(g))\{z\}$ is a translation covariant, local field.
By the extension theorem (\cite[Theorem 3.2]{BK}), we have a unique GVA structure on $F(g)$ satisfying $Y(b,z)=b(z)$.

The GVA $F(g)$ just constructed is an initial object in $\cG(g)$ as follows.
Let $V$ be an object of $\cG(g)$.
Consider the  surjective linear map $\pi:A\ra V$ defined by
$\pi(b(n_1)\cdots b(n_r))= b(n_1)\cdots b(n_r)\bm1$.
It follows by \eqref{bzeroalg} and \eqref{loce} that it induces a surjection 
$\bar\pi:F(g)\ra V$.
The map $\bar\pi$ preserves the vacuum vector, translation operator and generator $b$.
We have $\bar\pi(Y(b,z)v)=Y(\bar\pi (b),z)\bar\pi(v)$ for any $v\in F(g)$.
Since $b\in F(g)$ generates $F(g)$, we have $\bar\pi(Y(u,z)v)=Y(\bar\pi (u),z)\bar\pi(v)$ for all $u,v\in F(g)$.
Therefore, the map
$\bar\pi$ is a morphism in $\cG(g)$, which completes the proof.

%

\section{Free representations and RR exact sequences}\label{sec:freemod}
In this section, we introduce the free representations $\M m$ and RR exact sequences among the free representations over free GVAs.
Let $g$ be a complex number.

\subsection{Free representations}
Let $V$ be an object of $\cG(g)$.
Let $m$ be a complex number.
A {\em free GV module}  of order $-m$ over $V$
is a $\ZZ$-graded cyclic module $\M m=\bigoplus_{n\in\ZZ} \M m^n$ freely generated by
a formal element $v_m\in \M m^0$ such that 
\begin{equation}\label{mmvacum}
b(-m+n)v_m=0 \quad\mbox{for all} \quad n\in\CC\setminus \ZZ_{<0}.
\end{equation}
Here the associated biadditive map $\Delta_1:\ZZ\times \ZZ\ra \CC/\ZZ$ is defined by 
$\Delta_1(r,s)=-m-grs+\ZZ$.
In particular, if $V$ is a vertex algebra and $m\in\ZZ$, then $\M m$ is a usual vertex algebra module.
More precisely, a free GV module $\M m$ is a universal module
in the sense that if $M$ is a 
cyclic $\ZZ$-graded $V$-module generated by a
vector $v_m\in M^0$ satisfying \eqref{mmvacum}, then there is a unique surjective homomorphism $\M m\ra M$ sending $v_m$ to $v_m$.

The following proposition is proved in the same way as \Cref{propalg}.

\begin{proposition}\label{propmodgen}
Let $M$ be a $V$-module as in the above universality statement.
The following monomials span $M$:
\begin{equation}\label{basisfm}
b(-n_r)\cdots b(-n_1)v_m\quad (r\geq0, 
n_r\geqc g n_{r-1}\geqc g \cdots\geqc g n_1\geqc m 1).
\end{equation}
Furthermore,
\begin{equation}\label{bzero}
b(-n_r)\cdots b(-n_1)v_m=0\quad \mbox{if }
n_1+\cdots+n_r\not\in \frac g2 r(r-1)+mr+r+\Znn.
\end{equation}
\end{proposition}

Consider the submodule $I$ of $\M m$ generated by the element
$w=b(-m-1)v_m$.
We see from \eqref{bzero} that $b(-m-g+n)w=0$ for any
$n\not\in\CC\setminus \Zn$.
It follows that there is a surjective homomorphism $\M{m+g}\ra I$
sending $v_{m+g}$ to $w$.
Consider the quotient module $\M m/I$.
Since $b(-m-1)[v_m]=[w]$ is zero in $\M m/I$, we have an isomorphism 
$\M m/I\cong \M{m+1}$ sending $[v_m]$ to $v_{m+1}$ by the universality.
As a result, we have an exact sequence of the form
\begin{equation}\label{fundes}
\M {m+g}\xrightarrow\iota \M m\xrightarrow\pi \M{m+1}\ra 0.
\end{equation}
We will show in the following that $\iota$ is injective if $V=F(g)$.

\subsection{Free representations over $F(g)$}\label{subsecfreefg}
Let $m$ be a complex number and consider a free GV module
$\M m$ over $F(g)$.


%

The following proposition may be proved in the same way as the assertion that \eqref{basisfalg} form a basis of $F(g)$
(see \Cref{secps} for more details).

\begin{proposition}\label{propmod}
The monomials \eqref{basisfm} form a basis $\cB(m)$ of the free $F(g)$-module $\M m$.
\end{proposition}

The combinatorial basis $\cB(m)$  of $\M m$ implies that $b(-m-1)v_m\neq 0$.
Note that we have $\M0\cong F(g)$ as $F(g)$-modules.

The {\em character} of $\M m$ is the series
$\chi_{g;m}(z,q)=\sum_{r,d}\dim(\M m_{r,d})z^rq^d$,
where the bigrading is defined by the condition that
$b(-n_r)\cdots b(-n_1)v_m\in \M m_{r,n_1+\cdots+n_r}$.
We have $\chi_{g;m}(z,q)=F_g(zq^{1-g/2+m},q)$.
Let us write $\M m_{(r)}=\bigoplus_{d\in\CC}\M m_{r,d}$.
We have 
\begin{align}\label{fmcharge}
\M m=\bigoplus_{r\in\Znn}\M m_{(r)},
\quad \M m_{(0)}=\CC v_m.
\end{align}
By considering a relation between grading \eqref{fmcharge} and  the charge grading of lattice GVAs, we  call \eqref{fmcharge} a {\em charge grading} on $\M m$.
We say that $v_m\in \M m_{(0)}$ is a unique {\em lowest charge vector} of $\M m$
up to scalar multiples.

We may also define a normalized bigrading and the associated character as follows.
Consider the algebra $F(g;C,D)=F(g)$ with the normalized bigrading as in \Cref{subsecfreegva} with $C\in \CC^\times$ and $D\in\CC$.
Let $S$ and $E$ be complex numbers.
We denote by $\M {m;S,E}$ the $F(g;C,D)$-module $\M m$ 
equipped with the normalized bigrading $\M{m;S,E}=\bigoplus_{r,d}\M m_{r,d}$ such that
$b(-n_r)\cdots b(-n_1)v_m\in \M m_{S+Cr,E+d}$ with $d=n_1+\cdots+n_r-r+rD$.
The associated character $\chi_{g;m}^{(C,D;S,E)}(z,q)$ coincides with
$
z^S q^E F_g(z^\chb q^{D-g/2+m},q).
$


\subsection{The RR exact sequences}\label{subsecrr}


We have the following  {\em RR exact sequence} for free GV modules over free GVAs.
Let $g,m,\cwb,\chm,\cwm$ be complex numbers and $\chb$ a non-zero number.

\begin{theorem}\label{prop:fund}
There is a short exact sequence
$$
0\ra \M {m+g}\xrightarrow\iota \M m\xrightarrow\pi \M {m+1}\ra 0
$$
of $F(g)$-modules,
where the morphisms are given by $\iota(v_{m+g})=v_m$ and $ \pi(v_m)=v_{m+1}$.
It may be seen as a grading preserving exact sequence
$$
0\ra \M{m+g;\chm+\chb,\cwm+m+\cwb}\ra \M{m;\chm,\cwm}\ra \M{m+1;\chm,\cwm}\ra 0
$$
of bigraded $F(g;\chb,\cwb)$-modules.
\end{theorem}

\begin{proof}
We see that $\iota$ in \eqref{fundes} sends the basis $\cB(m+g)$ of $\M {m+g}$ injectively
to $\cB(m)$. Namely, the image of $b(-n_r)\cdots b(-n_1)v_{m+g}\in \cB(m+g)$  is 
the element
$b(-n_r)\cdots b(-n_1)b(-m-1)v_m$ of $\cB(m)$.
Thus, $\iota$ is an injection.
The rest is clear.
\end{proof}

We now have a recursion of the characters from the second statement in 
\Cref{prop:fund}:
$$
F_g(z^\chb q^{m+\cwb-g/2},q)=F_g(z^\chb q^{m+\cwb-g/2+1},q)+z^\chb q^{m+\cwb}F_g(z^\chb q^{m+\cwb+g/2},q),
$$
 where the common factor $q^Ez^\chm$ in the characters is  divided out from the formula.
By replacing $z^\chb q^{m+\cwb-g/2}$ with $z$, we recover 
\eqref{genrrbasic}.

\section{Finite RR exact sequences}\label{secfin}
In this section, we introduce a finite-dimensional vertex algebra 
$\fda kp$
and consider their free modules $\fd km$. We construct  RR exact sequence \eqref{finrrintro}, which lifts \eqref{finrrbasic}.

\subsection{The finite-dimensional vertex algebra $\fda kp$}
Let $p$ be a positive integer and consider a free vertex algebra 
$F(p)$ generated by a free generator $a\in F(p)^1$.
Let $k$  be an integer.
Consider a quotient vertex algebra
$\fda kp=F(p)/\sum_{n\geq 0}a(-k-n)F(p)$.

Note that since $F(p)$ is a commutative vertex algebra, we may say that
$\fda kp$ is freely generated by an element $a$ such that
$a(-m+n)a=0$ and $a(-k-n)a=0$ for any $n\geq0$.

We see that the following monomials form a basis 
$\cB_k$ of $\fda kp$:
\begin{equation}\label{eqn:basisfda}
a(-n_r)\cdots a(-n_1)\bm1\quad (r\geq0, n_1,\ldots,n_r\in\ZZ 
\mbox{ such that }
k\geqc 1 n_r\geqc p n_{r-1}\geqc p \cdots\geqc p n_1\geqc 0 1),
\end{equation}
(see \Cref{subseccoinvva} for a proof.)
In particular, if $k\leq 1$ then $\fda kp$ is a one-dimensional trivial vertex algebra $\CC\bm1$.

For example, let us consider $p=2$ and write $a=e$. Then we have the following bases of $\fda 02, \fda 12, \ldots, \fda 52$,
respectively:
\begin{align*}
&\{\bm1\},\quad \{\bm1\}, \quad \{\bm1,e(-1)\bm1\}, 
 \quad \{\bm1,e(-1)\bm1,e(-2)\bm1\}, \quad
  \{\bm1,e(-1)\bm1,e(-2)\bm1,e(-3)\bm1,e(-3)e(-1)\bm1\}, \\
 & \{\bm1,e(-1)\bm1,e(-2)\bm1,e(-3)\bm1,e(-3)e(-1)\bm1,
 e(-4)\bm1,e(-4)e(-1)\bm1,e(-4)e(-2)\bm1\}.
\end{align*}
One notices that the number of elements of each bases are Fibonacci numbers, which will be considered in later sections.

The bigrading and character  $\chi_{k,p}(z,q)$ of $\fda kp$ are defined in the same way as those of  $F(p)$. We have
$\chi_{k,p}(z,q)=1$ if $k\leq 1$.
Suppose $k\geq 2$ and let $k$ and $d$ be non-negative integers.
Since \eqref{eqn:basisfda} belongs to $\fda kp_{r,n_1+\cdots+n_r}$, the dimension of  $\fda kp_{r,d}$ is counted as a restricted partition number.
More precisely,
\begin{align*}
&\dim(\fda kp_{r,d})
=
\#\{(n_1,\ldots,n_r)\,|\,n_1+\cdots+n_r=d,
k\geqc 1 n_r\geqc p n_{r-1}\geqc p \cdots\geqc p n_1\geqc 0 1\}
\\
&\quad=
\#\Biggl\{(m_1,\ldots,m_r)\,\Biggl|\,
\begin{aligned}
&m_1+\cdots+m_r=d-r-\frac {pr}2(r-1),\\
&k-p(r-1)-1\geq m_r\geq m_{r-1}\geq \cdots\geq m_1\geq 0 
\end{aligned}
\Biggr\}.
\end{align*}
It implies that $\sum_{d\geq 0}\dim(\fda kp_{r,d})q^d=
q^{pr^2/2-pr/2+r}\qbinom {k+p-2-(p-1)r}r_q$ for any $r\geq0$ (see, e.g., \cite{A}).
Here, $\qbinom nr_q$ is the $q$-binomial coefficient
$$
\displaystyle\qbinom nr_q=
\begin{cases}
\frac{(q)_n}{(q)_r(q)_{n-r}}& (n\geq r\geq0),\\
0&(\mbox{otherwise}).
\end{cases}
$$ 
As a result, we have $\chi_{k,p}(z,q)=F_{p,k+p-2}(zq^{1-p/2},q)$.

Let $C$ be a non-zero complex number and $D$ a complex number.
We have a normalized bigrading $\fda kp=\bigoplus_{r,d}\fda kp_{r,d}$ on $\fda kp$ such that 
\eqref{eqn:basisfda} belongs to $\fda kp_{Cr,n_1+\cdots+n_r-r+Dr}$.
If we equip $\fda kp$ with this bigrading, we write $\fda kp=\fda k{p;C,D}$.
We have the normalized associated character 
$\chi_{k,p}^{(C,D)}(z,q)=F_{p,k+p-2}(z^Cq^{D-p/2},q)$.

\subsection{Free representations over $\fda kp$}\label{subsecfreefda}
Let $m$ be an integer and $\fd km$ denote the free representation of order $-m\in\ZZ$ over $\fda kp$ freely generated by $v_{k,m}$.
It has a combinatorial basis $\cB(k,m)$ consisting of monomials of the form
\begin{equation}\label{eqn:basisfd}
a(-n_r)\cdots a(-n_1)v_{k,m}\quad (r\geq0, n_1,\ldots,n_r\in\ZZ 
\mbox{ such that }
k\geqc 1 n_r\geqc p n_{r-1}\geqc p \cdots\geqc p n_1\geqc 0 m+1),
\end{equation}
(see \Cref{subseccoinvva} for a proof.)
In particular, the character of $\fd km$ has the form
$$
\chi_{k,p;m}(z,q)=\begin{cases}
1&(k\leq m+1)\\
F_{p,k+p-2-m}(zq^{1-p/2+m},z)&(k\geq m+2)
\end{cases}.
$$

Consider $\fda kp=\fda k{p;C,D}$ with a normalized bigrading.
Let $S$ and $E$ be complex numbers.
We write $\fd km=\fd k{m;S,E}$ if we equip $\fd km$ the bigrading such that \eqref{eqn:basisfd} belongs to 
$\fd km_{S+Cr, E+n_1+\cdots+n_r-r+Dr}$.
The associated character has the form
$\chi_{k,p;m}^{(C,D;S,E)}(z,q)=z^S q^E F_{p,k+p-2-m}(z^Cq^{D-p/2+m},q)$.

\subsection{RR exact sequences for $\fd km$}\label{subsecfinrr}
We now have RR exact sequences for $\fd km$.

\begin{proposition}\label{propfibrr}
If $k\geq m+2$, then there is an exact sequence of the form
\begin{align}\label{fibrr}
0\ra \fd k{m+p}\xrightarrow\iota \fd km\xrightarrow\pi \fd k{m+1}\ra0,
\end{align}
with the morphisms uniquely determined by
$\iota(v_{k,m+p})=a(-m-1)v_{k,m}$ and $\pi(v_{k,m})=v_{k,m+1}$.
It may be seen as a grading preserving exact sequence
\begin{align*}
&0\ra \fd k{m+p;\chm+\chb,E+m+\cwb}\ra \fd k{m;\chm,\cwm}
 \ra \fd k{m+1;\chm,\cwm}\ra 0,
\end{align*}
among bigraded $\fda k{p;C,D}$-modules.
\end{proposition}

\begin{proof}
We see that $\iota$ in \eqref{fundes} with $V=\fda kp$ sends the basis $\cB(k,m+g)$ of $\fd k {m+g}$ injectively
to $\cB(k,m)$. Namely, the image of $a(-n_r)\cdots a(-n_1)v_{k,m+g}\in \cB(k,m+g)$  is 
the element
$a(-n_r)\cdots a(-n_1)a(-m-1)v_{k,m}$ of $\cB(k,m)$.
Thus, $\iota$ is an injection.
The rest is clear.
\end{proof}

Now, by taking the characters of the latter sequence in 
\Cref{propfibrr}, we have the recursion relation
$$
F_{p,k+p-2-m}(z^\chb q^{\cwb-p/2+m},q)=
F_{p,k+p-1-m}(z^\chb q^{\cwb-p/2+m+1},q)
+z^\chb q^{m+\cwb}F_{p,k-2-m}(z^\chb q^{\cwb+p/2+m},q),
$$
 where the common factor $q^Ez^\chm$ in the characters is  divided out from the formula.
By replacing $z^\chb q^{m+\cwb-p/2}$ with $z$, we recover 
\eqref{genrrbasic} with $\ell=k-m-2$.

We remark that \eqref{fibrr} can be induced from \Cref{prop:fund} as follows.
We first note that $\fd km$ is a module over $F(p)$ and we have
an isomorphism $\fd km\cong \M m/J_k(\M m)$ as $F(p)$-modules, 
where $J_k(\M m)=\sum_{n\geq0}a(-n-k)\M m$.
 Since the morphisms in the RR exact sequence in  \Cref{prop:fund} are $F(p)$-module homomorphisms, they send the submodule $J_k(\M{m+p})$ into $J_k(\M m)$ and $J_k(\M m)$ into $J_k(\M {m+1})$.
Therefore, the complex of the form \eqref{fibrr} is induced.
By comparing the bases, we see that the complex is exact.

We remark that spaces with the same or related characters 
 appear in \cite{FKLMM,FF,Ke2} and others.

\section{Preliminaries on applications}\label{secpreapp}

We start applications of our RR exact sequences.
In this section, we recall necessary ingredients for our applications 
of RR exact sequences.

\subsection{Lattice generalized vertex algebras}\label{subseclattice}
Let $\ah=\CC h$ be a one-dimensional vector space equipped with 
a non-zero symmetric bilinear form $(\cdot,\cdot):\ah\ra \ah$.
Consider the associated Heisenberg Lie algebra $\hh=\ah\otimes_{\CC}\CC[t,t^{-1}]\oplus \CC K$ with the Lie bracket defined by linearly extending the assignments
$$
[h(n),h(m)]=n(h,h)\delta_{n+m,0}K,\quad [K,\hh]=0\qquad (n,m\in\ZZ).
$$
For any $\lam\in\CC$, we have a Fock representation $M(1,\lam)=U(\hh)\otimes_{U(\ah[t]\oplus \CC K)}\CC_\lam$, where $\CC_\lam$ is the one-dimensional $\ah[t]\oplus \CC K$-module defined by $h(0)=\lam\cdot\mathrm{id}$, $h(n)=0$ ($n> 0$) and $K=\id$.
We write $e^\lam=1\otimes 1\in M(1,\lam)$.

Consider the  $\ah$-graded module
$
V_\ah=\bigoplus_{\gamma\in\ah}V_\ah^\gamma$ with $V_\ah^\gamma=M(1,\gamma)
$ 
($\gamma\in\ah$)
over $\hh$ and equip it with the vacuum vector $\bm1= e^0$ and translation operator $\partial$ defined by $\partial\bm1=0$ and
$[\partial,h(n)]=-nh(n-1)$.
We equip $\ah$ with a phase $\eta$ defined by $\eta_{\gamma,\delta}=e^{\pi\sqrt{-1}(\gamma,\delta)}$ ($\gamma,\delta\in\ah$).
 The associated biadditive map has the form $\biadd(\gamma,\delta)=-(\gamma,\delta)+\ZZ$.
Then we have a unique GVA structure on $V_\ah$ by extending the assignment
$$
Y(e^\gamma,z)=\exp\left(\sum_{n\in\ZZ_{<0}}\frac{\gamma(n)}{-n}z^{-n}\right)
\exp\left(\sum_{n\in\ZZ_{>0}}\frac{\gamma(n)}{-n}z^{-n}\right)
  z^{\gamma(0)}e_{\gamma}\qquad (\gamma\in\ah),
$$
where $z^{\gamma(0)} e^\lam=z^{(\gamma,\lam)}e^\lam$ and $e_\gamma g e^\lam=ge^{\gamma+\lam}$ ($g\in U(\hh)$, $\gamma,\lam\in\ah$).
 It is called a GVA associated to the vector space $\ah$ \cite{BK}.

 In particular, we have
 \begin{equation}\label{loclat}
 e^\gamma(-(\gamma,\delta)-1)e^\delta=e^{\gamma+\delta},\quad
 e^\gamma(-(\gamma,\delta)+n)e^\delta=0\quad
 (n\in\CC\setminus\Zp),
 \end{equation}
 which implies that the locality order of  $(e^\gamma,e^\delta)$ 
 is $N_0=-(\gamma,\delta)$.
 
 For any subset $A\subset \ah$, we set $V_A=\bigoplus_{\gamma\in A} V_\ah^\gamma$.
 Suppose that $A\subset \ah$ is a rank one {\em lattice}.
Then the subspace $V_A$ is an $A$-graded GVA inside  $V_\ah$, called the {\em lattice GVA} associated to the rank one lattice $A$ \cite{DL}.
The grading 
$V_\ah=\bigoplus_{\lam\in\ah}(V_\ah)^\lam$ is called 
the {\em charge grading} of $V_\ah$.

Suppose that the rank one lattice $A=\ZZ\gamma$ is an integral lattice, equivalently, $(\gamma,\gamma)\in\ZZ$. 
 Then $V_A$ is a usual 
lattice vertex (super)algebra. 
We then consider the dual lattice $A^\circ=\ZZ\gamma^\circ$ of $A$ with $\gamma^\circ=\gamma/(\gamma,\gamma)$.
The inequivalent irreducible vertex algebra modules over $V_A$
are isomorphic to the submodules $V_{A+i\gamma^\circ}$ with $0\leq i\leq |(\gamma,\gamma)|-1$.

 \subsection{The Virasoro vector and bigradings}\label{subsecvir}
 Consider the element $\omega=\frac1{2(h,h)}h(-1)^2\bm1\in V_\ah^0$ and write
$Y(\omega,z)=\sum_{n\in \ZZ}L_nz^{-n-2}$. It induces
an action of the Virasoro algebra $\mathrm{Vir}=\bigoplus_{n\in\ZZ}\CC L_n\oplus\CC C$ on $V_\ah$.
We see that $C$ acts as 1 (central charge 1) and $L_{-1}=\partial$. 
Moreover, $L_0$ acts semisimply on $V_\ah$ so that we have the decomposition of $V_\ah$ by
{\em conformal weights}:
$
V_\ah=\bigoplus_{d\in\CC}(V_\ah)_d
$
with $
(V_\ah)_d=\{v\in V_\ah\,|\,L_0v=dv\}.
$
We see that
$$
L_0e^\gamma=\frac{(\gamma,\gamma)}2e^\gamma,\quad [L_0,h(n)]=-nh(n),\quad [L_0,e^\gamma(n)]=
\Bigl(-n-1+\frac{(\gamma,\gamma)}2\Bigr)e^\gamma(n).
$$
Since the $L_0$-grading is compatible with the charge grading, we  have a bigrading
$$
V_\ah=\bigoplus_{\gamma\in\ah,d\in\CC}(V_\ah)^\gamma_d,\quad
(V_\ah)^\gamma_d=V_\ah^\gamma\cap (V_\ah)_d.
$$
on $V_\ah$. 

For any rank one lattice $A\subset\ah$, the lattice GVA $V_A$  carries the induced gradings
$$
V_A=\bigoplus_{\gamma\in A}V_A^\gamma=\bigoplus_{d\in\CC}(V_A)_d
=\bigoplus_{\gamma\in A,d\in \CC}(V_A)^\gamma_d.
$$
The lattice GVA $V_A$ equipped with the Virasoro vector $\omega\in V_A$ is a {\em conformal} GVA.
If $A$ is a positive-definite integral lattice, then $V_A$ with $\omega$ is a {\em vertex operator algebra} (VOA).

\subsection{Restricted dual modules}\label{subsecdual}
In this section, we recall the notion of restricted dual modules.
Let $V=\bigoplus_{\gamma\in A}V^\gamma$ be a GVA
with a bigrading $V=\bigoplus_{r,d\in\CC}V_{r,d}$ such that the grading by $r$ refines the $A$-grading and that $V_{r,d}(n)V_{s,e}\subset V_{r+s,d+e-n-1}$
for any $r,s,d,e\in\CC$.
Suppose 
 that $V$ admits an action of the positive part  $\mathrm{Vir}_+=\bigoplus_{n=-1}^\infty \CC L_n$ of the Viraosoro algebra
such that $L_{-1}$ acts as $\partial$ and 
that $L_0$ acts as $d$ on $V_{r,d}$ for any $r,d$.
Let $M=\bigoplus_{\lam\in B}M^\lam$ be a $B$-graded $V$-module.
Suppose that $M$ admits a bigrading  $M=\bigoplus_{s,e\in\CC}M_{s,e}$ such that
 $V_{r,d}(n)M_{s,e}\subset M_{r+s,d+e-n-1}$.
Assume also that the grading by $s$ is a refinement of the $B$-grading,
so that we have a set $K_\lam\subset \CC$ for each $\lam\in B$
such that $M^\lam=\bigoplus_{s\in K_\lam,e\in\CC} M_{s,e}$.
Suppose that each bigraded piece $M_{s,e}$ is  finite-dimensional.
Consider the restricted dual space
$M^\vee=\bigoplus_{s,e}(M_{s,e})^*$ of $M$.
The space $M^\vee$ is $(-B)$-graded with $(M^\vee)^{-\lam}=\bigoplus_{s\in K_\lam,e\in\CC}(M_{s,e})^*$.
Then $M^\vee$ is a $V$-module with the transpose action 
$$
\langle Y^t(u,z)\varphi,v\rangle=\langle \varphi, Y(e^{zL_1}e^{\pi \sqrt{-1}L_0}z^{-2L_0}u,z^{-1})v\rangle\quad (u\in V, \varphi\in M^\vee, v\in M),
$$
called the {\em restricted dual module} of $M$ over $V$.
The module $M^\vee$ naturally carries a bigrading 
 with bigraded pieces  $(M^\vee)_{s,e}=(M_{-s,e})^*$ $(s,e\in\CC)$.
The associated characters of $M^\vee$ and $M$ are related as
$
\ch[M^\vee](z,q)=\ch[M](z^{-1},q),
$
where $\ch[M](z,q)=\sum_{s,e} \dim(M_{s,e})z^sq^e$.

Let us consider the restricted dual modules of  modules inside $V_\ah$.
We consider the $\ah$-grading on $V_\ah$ as a $\CC$-grading by $r$ in the above.
Let $V$ be a bigraded GVA inside $V_\ah$ closed under the positive part of the Virasoro action by $\omega$.
Let $M$ be a bigraded submodule of $V_\ah$ over $V$.
We then have the restricted dual module $M^\vee$.
Let $\gamma$ be an element of $\ah$ and suppose that $e^\gamma\in V$.
Since $e^\gamma$ is of conformal weight $d=(\gamma,\gamma)/2$ 
and $L_1e^\gamma=0$, we have 
\begin{equation}\label{contralattice}
 \langle {}^t(e^\gamma(n)) \varphi,v\rangle =\langle \varphi,
  e^{\pi\sqrt{-1}(\gamma,\gamma)/2} e^\gamma(-n+(\gamma,\gamma)-2)v\rangle
 \quad (\varphi\in M^\vee, v\in M, n\in\CC).
 \end{equation}
 
Consider $M=V_\ah$. Since $V_\ah$ admits a non-degenerate invariant bilinear form,
we have a $V$-module isomorphism $f:(V_\ah)^\vee\xrightarrow\sim V_\ah$ 
such that
$f({}^t(e^\gamma(n)) \varphi)= e^{\pi\sqrt{-1}(\gamma,\gamma)/2} e^\gamma(-n+(\gamma,\gamma)-2)f(\varphi)$.
Let  $B$ be a subset of $\ah$ and suppose that $V_B$ is a $V$-submodule of $V_\ah$.
We then have $(V_B)^\vee \cong V_{-B}$
as $V$-modules.

\subsection{Invariance and coinvariance construction}\label{subsecic}
We next introduce the invariance and coinvariance constructions.
For any $k\in\CC$ and series $u(z)=\sum_{n\in\CC}u_nz^{-n-1}$, we set 
$u(z)_{\geq k}=\sum_{n\geq 0}u_{k+n}z^{-k-n-1}$
and
$u(z)_{\leq k}=\sum_{n\geq 0}u_{k-n}z^{-k+n-1}$.
(We will consider the cases where $u_{k+ n}=0$ unless $n\in\ZZ$.)


Let $U$ and $W$ be vector spaces and
suppose that  $u(z)=\sum_{n\in\CC}u_nz^{-n-1}$ is an element of  $\mathrm{Hom}_\CC(U,W)\{z\}$.
The kernel  $\ker{u(z)}=\ker{u(z):U\ra W\{z\}}$ of $u(z)$  coincides with $\bigcap_{n\in\CC}\ker{u_n}$.
The cokernel of $u(z)$ is by definition 
$\coker{u(z)}=W/\sum_{n\in\CC}\mathrm{Im}(u_n)$.

Let $V$ be a GVA and $v$ a homogeneous element of $V$.
Let $U$ be a sub\,GVA of $V$ commuting with $\{v\}$.
Let $M$ be a $V$-module and $R$ a graded $U$-submodule of $M$.
 We often consider the $U$-submodule
$$
\ker{Y(v,z)_{\geq k}|_R}=\bigcap_{n\geq0}
\ker{v(k+n)|_R:R\ra M},
$$ 
which we denote by $I_{v,k}(R)=I_k(R)=\ker{Y(v,z)_{\geq k}|_R}$.

Let  $T$ be a quotient $U$-module of $M$ with the
canonical surjection $\pi:M\ra T$.
The cokernel 
$$
\coker{\pi\circ Y(v,z)_{\leq -k}}
=\frac T{\sum_{n\geq0} \pi\circ v(-k-n)M}
$$
 is a $U$-module. 
  We denote $C_{v,k}(T)=I_k(T)=\coker{\pi\circ Y(v,z)_{\leq -k}}$.

\section{Lattice principal subspaces and RR exact sequences}
\label{secps}
In this section, we recall lattice principal subspaces and compare our
exact sequence \eqref{ourintro-} with those in \cite{CLM03} and others.
We take over the notations $V_\ah$, $\omega$ and others from \Cref{subseclattice}--\Cref{subsecvir}.

\subsection{The vacuum lattice principal subspaces}\label{subsecps}

Let $g$ be a non-zero complex number.
Let $\beta$ be an element of $\ah$ such that $(\beta,\beta)=g$
and set $L=\ZZ\beta$.
Consider the $L$-graded sub\,GVA $W_L$ of $V_L$ (and hence of $V_\ah$) generated by  $e^\beta$, which we call a (generalized) vacuum lattice principal subspace \cite{MP,Kaw15}.
The space $W_L$ is a bigraded subspace of $V_L$ of the form
$W_L=\bigoplus_{r\geq0, d\in\CC}(W_L)^{r\beta}_d$,
where
$$
(W_L)^{r\beta}_d=(W_L)^{r\beta}\cap (W_L)_d,\quad
(W_L)^{r\beta}=W_L\cap (V_L)^{r\beta},\quad
(W_L)_{d}=W_L\cap (V_L)_{d}.
$$
Moreover, $W_L$ is closed under $L_n$ with $n\geq-1$ although 
the Virasoro vector $\omega$ is not included in $W_L$.

Since the locality order of $e^\beta$ is $-g$, the GVA $W_L$ is an object of
$\cG(g)$ if we identify $e^\beta$ with the indeterminate $b$ and 
consider the $L$-grading of $W_L$ as a $\ZZ$-grading 
by identifying $r\beta$ with $r$.
As explained in \Cref{subsecgva}, we have an isomorphism $F(g)\xrightarrow{\sim} W_L$.
We often identify $F(g)$ with $W_L$.

Consider the algebra $F(g)=F(g;1,g/2)$ equipped with the 
normalized bigrading such that $b\in F(g)_{1,g/2}$.
Then the isomorphism $F(g;1,g/2)\cong W_L$ 
sends each homogeneous piece
$F(g)_{r,d}$  to $(W_L)^{r\beta}_d$ $(r\geq 0,d\in\CC)$.

Let $r$ be a positive integer.
Since $e^\beta(-(\beta,\gamma)-1)e^\gamma=e^{\beta+\gamma}$,
it follows that 
$$
e^{r\beta}=e^\beta(-g(r-1)-1)\cdots e^\beta(-2g-1)e^\beta(-g-1)\bm1
$$
and therefore, we have the embedding $W_{rL}\subset W_L$.
It then follows that we have the embedding $F(r^2g)\subset F(g)$
with $b'=b(-g(r-1)-1)\cdots b(-2g-1)b(-g-1)\bm1$, where $b$ and $b'$ are the free generators of $F(g)$ and $F(r^2g)$, respectively.
In particular, if $p$ is a positive integer, then
$F(p)$ is naturally embedded in $F(1/p)$.

Note that even if $g=0$, we have the isomorphism $F(0)\cong W_{\ZZ\gamma}$ sending $b$ to $e^\gamma$, where $\ZZ\gamma$ is a lattice with 
the zero bilinear form (see \cite{R}).

\subsection{General lattice principal subspaces}
Let us take over the notations from the last section.
Consider the cyclic $W_L$-module $W_L(m\beta^\circ)=W_L\cdot e^{m\beta^\circ}\subset V_{L+m\beta^\circ}$ generated by $e^{m\beta^\circ}$ with $\beta^\circ=\beta/g$, which we call a 
(generalized) lattice principal subspace \cite{MP,Kaw15}.

Let us identify $F(g)$ with $W_L$
 and consider
a free GV module $\M m$ freely generated by $v_m$.
Since it follows from \eqref{loclat} that $e^\beta(-m+n)e^{m\beta^\circ}=0$ for any $n\in\CC\setminus\Zn$, we have a surjection $\chi_m:\M m\ra W_L(m\beta^\circ)$ sending $v_m$ to $e^{m\beta^\circ}$, by the universality of $\M m$.
Here we identify the $(m\beta^\circ+L)$-grading on $W_L(m\beta^\circ)$
with the defining $\ZZ$-grading on $\M m$ via $m\beta^\circ +r\beta\mapsto r$.
We know that elements of the form \eqref{basisfm} form a  basis of  $W_L(m\beta^\circ)$
\cite{Geog,MP,Kaw15}.
Moreover, \eqref{lemfund} implies that the elements \eqref{basisfm} in $\M m$ span $\M m$.
As a result, we have the following lemma.

\begin{lemma}\label{lemwm}
The map $\chi_m: \M m\ra W_L(m\beta^\circ)$ is an isomorphism.
\end{lemma}

Let us consider bigradings. 
The space  $W_L(m\beta^\circ)$ is a bigraded subspace of $V_\ah$.
Let us identify $F(g;1,g/2)$ with $W_L$ and consider a  free GV module $\M{m;g^{-1}m,g^{-1}m^2/2}$ with a normalized bigrading.
The isomorphism
$
\chi_m:\M{m;g^{-1}m,g^{-1}m^2/2}\ra W_L(m\beta^\circ)
$ 
sends $\M m_{r,d}$ to $W_L(m\beta^\circ)^{r\beta}_d$ for any $r\in g^{-1}m+\Znn,d\in\CC$.

 Note that even if $g=0$, we may realize $\M m$  inside a GVA associated to a 2-dimensional vector space with a non-degenerate bilinear form. We can show that \eqref{basisfm} form a basis of $\M m$. The proof is easy and omitted.
 In the rest of the paper, we often assume that $g$ is a  non-zero complex number.

\subsection{The RR exact sequence for lattice principal subspaces}
\label{subsecRRprin}
Let $g$ be a non-zero complex number and $m$ a complex number.
By using the isomorphism $\chi_m:\M m\cong W_L(m\beta^\circ)$, we have  
 an isomorphism of exact sequences
 \begin{equation*}
  \xymatrix{
    0\ar[r]& \M {m+g} \ar[r]^\iota \ar[d]_{\chi_{m+g}} & \M m\ar[r]^\pi \ar[d]_{\chi_m}&\M {m+1}\ar[r]\ar[d]_{\chi_{m+1}}&0\\
   0\ar[r] &W_L((m+g)\beta^\circ)
   \ar[r]^{f_1}&W_L(m\beta^\circ) \ar[r]^{f_2} &W_L((m+1)\beta^\circ)\ar[r]& 0,
  }
\end{equation*}
with morphisms $f_1$ and $f_2$,
where the upper sequence is our RR exact sequence in \Cref{prop:fund}.
We determine $f_1,f_2$ explicitly.

We have $f_1(e^{(m+g)\beta^\circ})=\chi_m\circ \iota\circ \chi_{m+g}^{-1}(e^{(m+g)\beta^\circ})=\chi_m\circ\iota(v_{m+g})=\chi_m(b(-m-1)v_m)=e^\beta(-m-1)e^{m\beta^\circ}=e^{(m+g)\beta^\circ}$.
Therefore,
the injection $f_1$ is nothing but the natural embedding.

The surjection $f_2$ is described as follows.
Since $(\beta,\beta^\circ)=1$, \eqref{commrel} implies the anticommutation relation
$
e^\beta(n)e^{\beta^\circ}(k)=-e^{\beta^\circ}(k)e^\beta(n)
$
$(n,k\in\CC)$.
We twist the operator $e^{\beta^\circ}(k)$ so that we have a genuine commutation relation.
Take a complete system of representatives $\Lam\subset\ah$
 of the cosets in $\ah/\ZZ\beta$ such that $m\beta^\circ\in \Lam$.
 Let $\varepsilon:V_\ah\ra V_\ah$ 
be a linear isomorphism defined by 
$\varepsilon(v)=(-1)^r v$ ($v\in V_\ah^{\lam+r\beta}$, $\lam\in\Lam$).
It then follows that the operators $e^{\beta^\circ}(k)\circ \varepsilon$  ($k\in\CC$) 
genuinely commute with $e^\beta$:
we have
$
e^\beta(n)\circ (e^{\beta^\circ}(k)\circ \varepsilon_\Lam)=(e^{\beta^\circ}(k)\circ \varepsilon_\Lam)\circ e^\beta(n)$ for any $n,k\in\CC$.
In particular, the operator $\varphi=e^{\beta^\circ}(-g^{-1}m-1)\circ \varepsilon_m$ is a $W_L$-module homomorphism from  $W_L(m\beta^\circ)$ to $W_L((m+1)\beta^\circ)$ since we have
$\varphi.e^{m\beta^\circ}=e^{(m+1)\beta^\circ}$.
We show $f_2=\varphi$.
Since $e^{m\beta^\circ}\in W_L(m\beta^\circ)$ generates $W_L(m\beta^\circ)$ over $W_L$, we only have to show $\varphi (e^{m\beta^\circ})=f_2(e^{m\beta^\circ})$.
The left-hand side is 
$\varphi(e^{m\beta^\circ})=e^{(m+1)\beta^\circ}$.
The right-hand side is 
$
f_2(e^{m\beta^\circ})=f_2\circ \chi_m(v_m)=
\chi_{m+1}\circ \pi(v_m)=\chi_{m+1}(v_{m+1})=e^{(m+1)\beta^\circ},$
where $\pi$ is the surjection in the RR exact sequence.
Thus, we have $f_2=\varphi$.

In summary, we have induced the exact sequences
\begin{equation}\label{exactps2}
0\ra W_L((m+g)\beta^\circ)\hookrightarrow W_L(m\beta^\circ) \xrightarrow{e^{\beta^\circ}(-g^{-1}m-1)\circ \varepsilon} W_L((m+1)\beta^\circ)\ra 0.
\end{equation}
Note that the twist $\varepsilon$ does not affect kernels and cokernels. We omit $\varepsilon$ in the rest of this paper and there will be no confusions.

\begin{remark}\label{remclm}
Consider the case of $g=2$ and write $\beta=\alpha$.
The lattice $L$ is a root lattice of type $A_1$:  $L=A_1=\ZZ\alpha$, $(\alpha,\alpha)=2$.
Then \eqref{exactps2} with $m=0$ has the form
\begin{equation}\label{exacta1}
0\ra W_{A_1}(\alpha)\ra W_{A_1}(0)\ra W_{A_1}(\alpha/2)\ra0.
\end{equation}
Consider the operator $e_{\gamma}$ in \Cref{subseclattice}, which sends $W_L(\alpha/2)$ to  $W_{A_1}(\alpha/2+\gamma)$
bijectively.
By composing \eqref{exacta1} with the operator $e_{\alpha/2}$, we have  a linear exact sequence
\begin{equation}\label{clmrr}
0\ra W_{A_1}(\alpha/2)\ra W_{A_1}(0)\ra W_{A_1}(\alpha/2)\ra0.
\end{equation}
It is the lift of \eqref{genrrbasic} with $g=2$ introduced  in \cite{CLM03}.
Note that in \cite{CLM03}, they use \eqref{clmrr} to
obtain the fermionic character formula of $W_{A_1}(0)$ and $W_{A_1}(\alpha/2)$, which was first proved in \cite{SF}.
We have in the same way the rank one cases of exact sequences for lattice vertex algebras associated to positive lattices in \cite{PSW}. See also \cite{CalLM14}.
\end{remark}

\section{Invariance construction of principal subspaces}\label{secinvcoinv}

In this section, as an application of the RR exact sequences, we show that we can construct
$W_{L^\circ}(m\beta)$ by using the invariance construction $I_{e^{\beta},k}(\cdot)$ introduced in \Cref{subsecic}.
We also consider the restricted dual version of the construction.

Let $g$ be a non-zero complex number.
Consider the GVA $V_\ah$ and the element $v=e^\beta$. 
The subalgebra $W_{L^\circ}\subset V_\ah$ commutes with $e^\beta$,
since so does $e^{\beta^\circ}$ and $W_{L^\circ}$
is generated by $e^{\beta^\circ}$.

\subsection{Invariance construction}

Let $k$ and $m$ be  complex numbers.
Let us  consider the $W_{L^\circ}$-module
$$
I_k(W_{L^\circ}(m\beta))=\ker{(Y(e^\beta,z)_{\geq k})|_{W_{L^\circ}(m\beta)}}
=\bigcap_{n\geq0}
\ker{e^\beta(k+n):W_{L^\circ}(m\beta)\ra V_\ah}.
$$

Note that we have the inductive system
\begin{equation}\label{indsys}
W_{L^\circ}(m\beta)\subset W_{L^\circ}((m-g^{-1})\beta)\subset W_{L^\circ}((m-2g^{-1})\beta)\subset \cdots\subset 
W_{L^\circ}((m-ig^{-1})\beta)\subset \cdots
\end{equation}
and correspondingly,
$$
I_{k}(W_{L^\circ}(m\beta))\subset I_{k}(W_{L^\circ}((m-g^{-1})\beta))\subset I_{k}(W_{L^\circ}((m-2g^{-1})\beta))\subset \cdots
$$
Since $\bigcup_{i=0}^\infty W_{L^\circ}((m-ig^{-1})\beta)=V_{L^\circ+m\beta}$,
the inductive limit has the form $\varinjlim I_{k}(W_{L^\circ}((m-ig^{-1})\beta))=I_{k}(V_{L^\circ+m\beta})$,
which is also a  $W_{L^\circ}$-module.

Let us consider the $L^\circ$ version of \eqref{exactps2}:
\begin{equation}\label{exactgps}
0\ra W_{L^\circ}((m+g^{-1})\beta)\subset W_{L^\circ}(m\beta) \xrightarrow{e^{\beta}(-gm-1)} W_{L^\circ}((m+1)\beta)\ra 0.
\end{equation}
By using this, we have the following.

\begin{proposition}\label{propinv}
For any $k\in\CC$ and $m\in -g^{-1}k+g^{-1}\ZZ$,
the submodule  $I_k(W_{L^\circ}(m\beta))$ has the form
\begin{align*}
I_k(W_{L^\circ}(m\beta))=\begin{cases}
W_{L^\circ}(m\beta)&(gm\geqc0 -k)\\
W_{L^\circ}(-g^{-1}k\beta)& (gm\leqc0 -k).
\end{cases}
\end{align*}
In particular, we have $I_k(V_{L^\circ-k\beta^\circ})=W_{L^\circ}(-g^{-1}k\beta)$.
\end{proposition}

\begin{proof}
The latter statement follows from the former by taking the direct limit. 

We show the former assertion. We omit $0$ from the inequality $\geqc0$.
Suppose $gm\geq -k$. Since $e^\beta(-gm+n)e^{m\beta}=0$ for all $n\geq0$, we have $e^\beta(k+n)e^{m\beta}=0$ for each $n\geq 0$.
Since $e^\beta(n)$ and $e^{\beta^\circ}(m)$ commute with each other for any $n,m\in\CC$ and $W_{L^\circ}(m\beta)$ is generated by $e^{m\beta}$ over $W_{L^\circ}$, we have
 $I_k(W_{L^\circ}(m\beta))=W_{L^\circ}(m\beta)$.
Now, suppose $gm\leq -k$. 
The inclusion 
$I_k(W_{L^\circ}(m\beta))\supset W_{L^\circ}(-g^{-1}k\beta)$ follows from $e^\beta(k+n)e^{-g^{-1}k\beta}=0$ ($n\geq0$).
Therefore, it remains to show the opposite inclusion.
We fix $k$ and prove it by induction on $gm$. 
The case of $gm=-k$ is already done in the above.
Suppose $gm <-k$,
which implies $-gm-1\in k+\Znn$.
By \eqref{exactgps}, we have $\ker{e^\beta(-gm-1):W_{L^\circ}(m\beta)\ra V_\ah}=W_{L^\circ}((m+g^{-1})\beta)$.
It then follows that $I_k(W_{L^\circ}(m\beta))\subset W_{L^\circ}((m+g^{-1})\beta)$.
By the induction hypothesis, we have $I_k(W_{L^\circ}(m\beta))\subset W_{L^\circ}(-g^{-1}k\beta)$, which completes the proof.
\end{proof}
In particular, we have $I_k(V_{L^\circ})=W_{L^\circ}$,
which coincides with the rank one case of \cite[Theorem 7.2]{Kaw15}.

\subsection{Coinvariance construction}\label{subseccoinv}
We next describe the restricted dual version of the proposition.

Let $k$ and $m$ be  complex numbers.
Consider the isomorphism $(V_\ah)^\vee\cong V_\ah$ so that
$W_{L^\circ}(m\beta)^\vee$ is a quotient $W_{L^\circ}$-module of $V_\ah$ with the canonical surjection $\pi$.
Consider the $W_{L^\circ}$-module 
$$
C_k(W_{L^\circ}(m\beta)^\vee) 
=\coker{\pi\circ Y(e^\beta,z)_{\leq -k}}
=\frac {W_{L^\circ}(m\beta)^\vee}{\sum_{n\geq0}\pi\circ e^\beta(-k-n)V_\ah}.
$$

Note that we have the projective system
$$
W_{L^\circ}(m\beta)^\vee\twoheadleftarrow W_{L^\circ}((m-g^{-1})\beta)^\vee \twoheadleftarrow W_{L^\circ}((m-2g^{-1})\beta)^\vee\twoheadleftarrow \cdots
$$
and correspondingly,
$$
C_k(W_{L^\circ}(m\beta)^\vee)\twoheadleftarrow C_k(W_{L^\circ}((m-g^{-1})\beta)^\vee)\twoheadleftarrow C_k(W_{L^\circ}((m-2g^{-1})\beta)^\vee)\twoheadleftarrow \cdots.
$$
Since $\varprojlim W_{L^\circ}((m-ig^{-1})\beta)^\vee=V_{L^\circ+m\beta}^\vee$,
we have $\varprojlim C_k(W_{L^\circ}((m-ig^{-1})\beta)^\vee)=C_{k}(V_{L^\circ+m\beta}^\vee)$.

%
It is easy to verify that for any bigraded $W_{L^\circ}$-submodule $R$ of $V_\ah$, we have
\begin{equation}\label{dpcdp}
I_k(R)^\vee\cong C_{k-g+2}(R^\vee).
\end{equation}
as $W_{L^\circ}$-modules, since we have \eqref{contralattice}.

By taking the restricted duals of \Cref{propinv}, we have the
following.

\begin{proposition}\label{propinvdual}
For any $k\in\CC$ and $m\in -g^{-1}k+g^{-1}\ZZ$,
the quotient module $C_{k-g+2}(W_{L^\circ}(m\beta)^\vee)$   has the form
\begin{align*}
C_{k-g+2}(W_{L^\circ}(m\beta)^\vee)\cong\begin{cases}
W_{L^\circ}(m\beta)^\vee&(gm\geqc0 -k)\\
W_{L^\circ}(-g^{-1}k\beta)^\vee& (gm\leqc0 -k).
\end{cases}
\end{align*}
Moreover, $C_k(V_{L^\circ-k\beta^\circ}^\vee)\cong W_{L^\circ}(-g^{-1}k\beta)^\vee$ as $W_{L^\circ}$-modules.
\end{proposition}

We remark about a relation between our construction and  $\asltwo$ spaces of coinvariants
 $L_{1,0}^{N,\infty}(\mathfrak n)$ ($N\in\ZZ$) studied in \cite{FKLMM}
 and others.
 The space  $L_{1,0}^{N,\infty}(\mathfrak n)$  by definition coincides with the space $C_{e^\alp,N}(V_{A_1})$ if we identify $V_{A_1}$ with $L_{1,0}$ 
 as $\asltwo$-modules.
 Here $L_{1,i}$ denotes the basic highest weight representation over $\asltwo$ of highest weight $i$ ($i=0,1$).
 We see that  $C_{e^\alp,N}(V_{A_1})$ is a module over the vertex algebra
 $W_{A_1^\circ}^{(A_1)}=\bigoplus_{\gamma\in A_1}W_{A_1^\circ}^\gamma$,
 since so is $V_{A_1}$.
Moreover, 
 $C_{e^\alp,N}(V_{A_1})$ is naturally in $C_{e^\alp,N}(V_{A_1^\circ})$  the $W_{A_1^\circ}^{(A_1)}$-submodule 
$C_{e^\alp,N}(V_{A_1^\circ})^{(A_1)}$.
 It now follows by \Cref{propinvdual} that $L_{1,0}^{N,\infty}(\mathfrak n)\cong W_{A_1^\circ}(-N\alpha/2)^{\vee,(A_1)}$ as 
 $W_{A_1^\circ}^{(A_1)}$-modules, where we write $\alpha$ instead of $\beta$.
  In particular, we have the fermionic character formula of 
  $L_{1,0}^{N,\infty}(\mathfrak n)$
 by using the character formula of the principal subspaces of type $A_1^\circ$.

\section{Fibonacci exact sequences}\label{secfib}

In this section, we consider a lift of \eqref{fibbasic}.
Recall the finite-dimensional vertex algebra $\fda kp$ and its free module $\fd km$ with $p\geq0$ and  $k,m\in\ZZ$.
Recall that their characters are essentially generalized $q$-Fibonacci polynomials.

Notice that by definition $\fda kp=C_{a,k}(F(p))=C_k(F(p))$, where the construction $C_k(\cdot)$ is introduced in \Cref{subsecic}.
We also have $\fd km\cong C_k(\M m)$ as $F(p)$-modules.
Consider the subset $\cB(m)_{<k}$ of $\cB(m)$ consisting of elements of $\cB(m)$ in \Cref{propmodgen} (with $b=a$) such that $n_r<k$.
The image of $\cB(m)_{<k}$ coincides with $\cB(k,m)$.

The Fibonacci nature of the character of $\fda kp\cong \fd k0\cong C_k(\M 0)$ follows from the following recursion relation
of the basis $\cB(0)_{<k}$:
\begin{align*}
&\cB(0)_{<k}=\{v_0\} \quad (k\leq 1), \quad
&\cB(0)_{<k}=\cB(0)_{<k-1}\sqcup a(-k+1)\cB(0)_{<k-p}\quad (k\geq 2).
\end{align*}
We also have a similar recursion for $\cB(m)_{<k}$ for any $m\in\ZZ$.
It implies the following Fibonacci type exact sequences.

\begin{proposition}\label{propfib}
There is a short exact sequence
\begin{align}\label{fibrr2}
0\ra \fd{k-p}m\xrightarrow f \fd km\xrightarrow g \fd{k-1}m
\ra 0,
\end{align}
of $F(p)$-modules
with the morphisms uniquely determined by
 $
 f(v_{k-p,m})=a(-k+1)v_{k,m}$ and $g(v_{k,m})=v_{k+1,m}.
 $
It may be seen as a grading preserving exact sequence
\begin{align*}
&0\ra \fd{k-p}{m;\chm+\chb,E+k-2+\cwb}
\ra \fd k{m;\chm,\cwm}\ra 
\fd{k-1}{m;\chm,\cwm}\ra0
\end{align*}
of bigraded $F(p;C,D)$-modules
\end{proposition}

\begin{proof}
Consider the submodule $I$ of $\fd km$ generated by 
the vector $w=a(-k+1)v_{k,m}$.
It then follows that 
 we have an isomorphism $\fd km/I\cong \fd{k-1}m$.

To prove that there is an isomorphism $\fd{k-p}m\cong I$, 
we first show that  $a(-k+p-n)w=0$ for any $n\geq0$.
Since $a$ commutes with itself, we have 
$a(-k+p-n)w=a(-k+1)a(-k+p-n)v_{k,m}$.
We then apply \eqref{lemfund} with $b=a$, $g=p$, $s=-k+1$
and $t=-k+p-n$  and then
$$
a(-k+p-n)w=\sum_{j\geq0}c_j a(-k-n-j)a(-k+1+p+j)v_{k,m}
-\sum_{j\geq 0}c'_j a(-k-j)a(-k+p-n+1+j)v_{k,m}
$$ 
holds with certain numbers $c_j,c'_j$.
The right-hand side is zero since $k+n+j\geq k$ and $k+j\geq k$. Thus, we have $a(-k+p-n)w=0$.

Therefore, there is a surjective morphism $\fd{k-p}m\ra I$ which sends $v_{k-p,m}$ to
$w=a(-k+1)v_{k,m}$.
By comparing the bases, we see that it is injective.
The rest is clear.
\end{proof}

The sequence \eqref{fibrr2} is a  lift of the Fibonacci recursion \eqref{fibbasic}.
In \Cref{subsecswitch}, we observe that the restricted dual of \eqref{fibrr2} is
induced from the RR exact sequence in \Cref{prop:fund} with $g=1/p$.

\section{RF-type finite-dimensional modules}\label{secrf}

In this section, we introduce RF-type modules over $F(p)$,
which turn out to be isomorphic to the restricted
duals of $\fd km$ in \Cref{subsecswitch}.
We give combinatorial bases and construct RR exact sequences.


\subsection{Definitions}
Consider the free GVA $F(1/p)$ generated by an element $x$.
The free vertex algebra $F(p)$ is embedded in $F(1/p)$ 
with the identification
$a=x(-(p-1)/p-1)\cdots x(-2/p-1)x(-1/p-1)x$ (\Cref{secps}).
Moreover, $a$ and $x$ commute with each other.
Let $m$ be an integer.
We denote by $\Mcirc{m/p}$ a free GV module over $F(1/p)$ of order $-m/p$ freely generated  by an element $v_{m/p}$.

For any $i\in\{0,1,\ldots,p-1\}$, let us write $F(1/p)^{(i/p+\ZZ)}$
(resp., $\M m^{(i/p+\ZZ)}$) the subspace of $F(1/p)$ (resp., $\M m$) spanned by 
monomials of  length $i+pr$ with $r\in\Znn$.
Here, a monomial of the form $x(-n_r)\cdots x(-n_1)v_{m/p}$ has length $r$.
We  have the direct sum decomposition
$$
F(1/p)=\bigoplus_{i/p+\ZZ\in \frac1p\ZZ/\ZZ}F(1/p)^{(i/p+\ZZ)},\qquad 
\Mcirc{\frac mp}=\bigoplus_{i/p+\ZZ\in \frac1p\ZZ/\ZZ}\Mcirc{\frac mp}^{(i/p+\ZZ)}
$$
as $F(p)$-modules.
Note that the sub\,GVA $F(1/p)^{(\ZZ)}$ is a vertex algebra containing $F(p)$ 
and the above direct summands are all $F(1/p)^{(\ZZ)}$-modules.

Let $k$ and $m$ be integers.
Since $x$ and $a$ commute with each other, we
have  $F(p)$-modules
\begin{align*}
&\rf {k/p}{\frac mp}
=\coker{Y(x,z)_{\leq -k/p}|_{\Mcirc{m/p}^{(\ZZ)}}}
=\frac {\Mcirc{m/p}^{(1/p+\ZZ)}}{\sum_{n\in\ZZ_{\geq0}}x(-k/p-n)\Mcirc{m/p}^{(\ZZ)}},\\
&\ef {k/p}{\frac mp}
=\coker{Y(x,z)_{\leq -k/p}|_{\Mcirc{m/p}}}
= \frac {\Mcirc{m/p}}{\sum_{n\in\ZZ_{\geq0}}x(-k/p-n/p)\Mcirc{m/p}}.
\end{align*}
There is a direct sum decomposition 
\begin{equation}\label{efdecomp}
\ef{k/p}{\frac mp}=\bigoplus_{i=0}^{p-1}
\ef{k/p}{\frac mp,i},\quad 
\left(\ef{k/p}{\frac mp,i}=\ef{k/p}{\frac mp}^{((i+1)/p+\ZZ)}\right)
\end{equation}
as $F(p)$-modules.
Since we have
$$
\ef{k/p}{\frac mp,i}\cong
\frac {\Mcirc{m/p}^{((i+1)/p+\ZZ)}}{\sum_{n\in\ZZ_{\geq0}}x(-(k-i)/p-n)\Mcirc{m/p}^{(i/p+\ZZ)}},
$$
it follows that  $\ef{k/p}{m/p,0}\cong \rf{k/p}{m/p}$.

The character of $\rf {k/p}{m/p}$ is the series
 $\ch[\rf {k/p}{m/p}]=\sum_{r,d}(\dim(\rf {k/p}{m/p}_{r,d}))z^rq^d$,
 where 
$\rf {k/p}{m/p}_{r,d}$ is spanned by monomials 
$[x(-n_r)\cdots x(-n_1)v_m]$ with $n_1+\cdots+n_r=d$. 
Those of $\ef {k/p}{m/p}$ and $\ef {k/p}{m/p,i}$ are defined in the same way.

We consider the following normalized bigradings and associated characters.
Let $\chb,\cwb,\chb^\circ,\cwb^\circ$ be complex numbers
such that $\chb=p\chb^\circ\neq0$ and $\cwb=(\cwb^\circ+1/2)p-1/2$.
We consider the bigraded algebras $F(p;\chb,\cwb)$ and  $F(1/p;\chb^\circ,\cwb^\circ)=F(1/p;\chb^\circ,\cwb^\circ)$ with the normalized bigradings.
The embedding $F(p;\chb,\cwb)\subset F(1/p;\chb^\circ,\cwb^\circ)$ respects the bigradings.

Let $\chm,\cwm$ be complex numbers. 
Consider the bigraded module $\Mcirc{m/p;\chm,\cwm}$ over 
$F(1/p;\chb^\circ,\cwb^\circ)$.
Then we equip $\rf{k/p}{m/p}$, $\ef{k/p}{m/p}$ and $\ef{k/p}{m/p,i}$
with the bigradings induced from that on $\Mcirc{m/p;\chm,\cwm}$
and denote them by $\rf{k/p}{m/p;\chm,\cwm}$ and so on.
The natural embedding
 $\rf{k/p}{m/p;\chm,\cwm}\hookrightarrow \ef{k/p}{m/p;\chm,\cwm}$
respects the bigrading.
The associated character of $\rf{k/p}{m/p;\chm,\cwm}$ is denoted by $\ch[\rf{k/p}{m/p}](z,q)$ and so are the others.

In this section, we determine a basis $\ef{k/p}{m/p}$ and then deduce a basis for $\rf{k/p}{m/p}$. We then show RR exact sequences.

\subsection{Bases of EF- and RF-type modules}
To see the ideas, we start from  $m=0$.
We identify $\Mcirc0$ with $F(1/p)$ and consider $\ef {k/p}0$ as a quotient of $F(1/p)$.
We warn that $\ef {k/p}0$ is not a quotient GVA in general but always an $F(p)$-module.
Recall that we have  the following combinatorial basis $\cB^\circ$ of $F(1/p)$:
$
\cB^\circ=\{x(-n_r)\cdots x(-n_1)\bm1\,|\,r\geq0, 
n_r\geqc {1/p} \cdots\geqc {1/p} n_1\geqc 0 1\}.
$
If $k\leq p$, we  have $\ef{k/p}0=\CC[\bm1]$.
Suppose that $k>p$. We have the partition $\cB^\circ=\cB^\circ_{<k/p}\sqcup \cB^\circ_{\geq k/p}$, where $\cB^\circ_{<k/p}$ (resp., $\cB^\circ_{\geq k/p}$) consists of the monomials in $\cB^\circ$ with $n_r\in k/p+\QQ_{<0}$ (resp., $n_r\in k/p+\QQ_{\geq0}$).

\begin{theorem}\label{thmfibcircvac}
The images of the elements of  $\cB^{\circ}_{<k/p}$ 
form a $\CC$-basis of $\ef{k/p}0$.
\end{theorem}

A proof of the theorem is written in \Cref{proofthmfibcirc}.
We now give a basis of
$\ef{k/p}{m/p;\chm,\cwm}$ with $k,m\in\ZZ$ and $\chm,\cwm\in\CC$. 
We denote the combinatorial basis of $\Mcirc{m/p}$ by $\cB^\circ(m/p)$:
$$
\cB^\circ\left(\frac mp \right)=\{x(-n_r)\cdots x(-n_1)v_{m/p}\,|\,r\geq0, 
n_r\geqc {1/p} \cdots\geqc {1/p} n_1\geqc {m/p} 1\}.
$$
If $k\leq p+m$, we  have $\ef{k/p}{m/p}= \CC [v_{m/p}]$.
Suppose that $k>p+m$. We have the partition $\cB^\circ(m/p)=\cB^\circ(m/p)_{<k/p}\sqcup \cB^\circ(m/p)_{\geq k/p}$, where $\cB^\circ(m/p)_{<k/p}$ (resp., $\cB^\circ(m/p)_{\geq k/p}$) consists of the monomials in $\cB^\circ(m/p)$ with $n_r\in k/p+\QQ_{<0}$ (resp., $n_r\in k/p+\QQ_{\geq0}$).
We now have the following theorem.

\begin{theorem}\label{thmfibcirc}
The images of the  elements of $\cB^{\circ}(m/p)_{<k/p}$  form a $\CC$-basis of $\ef{k/p}{m/p}$.
\end{theorem}

The proof is similar to that of \Cref{thmfibcircvac} and omitted.

\subsection{Exact sequences for $\rf {k/p}{m/p}$}

In this section, we apply \Cref{prop:fund} to obtain  exact sequences among EF- and RF-type modules.
By \Cref{prop:fund} with $g=1/p$, we have an exact sequence
\begin{align}\label{efrrpre}
0\ra 
\Mcirc{\frac{m+1}p;\chm+\chb^\circ,E+\frac mp+\cwb^\circ}
\xrightarrow\iota
\Mcirc{\frac mp;\chm,\cwm}
\xrightarrow\pi
 \Mcirc{\frac{m}p+1;\chm,\cwm}\ra 0,
\end{align}
of $F(1/p)$-modules such that
 $
\iota(v_{(m+1)/p})=x(-m/p-1)v_{m/p}$ and 
$\pi(v_{m/p})=v_{m/p+1}$.

\begin{proposition}\label{propefrr}
If $k> m+p$, the exact sequence \eqref{efrrpre} induces an exact sequence of bigraded $F(p;\chb,\cwb)$-modules of the form
\begin{align}\label{efrr}
0\ra \ef {k/p}{\frac{m+1}p;\chm+\chb^\circ,E+\frac mp+\cwb^\circ}
\xrightarrow{\bar\iota} \ef  {k/p}{\frac mp;\chm,\cwm}
\xrightarrow{\bar\pi} \ef  {k/p}{\frac{m}p+1;\chm,\cwm}\ra 0,
\end{align}
where $\bar\iota$ sends an element $[x(-n_r)\cdots x(-n_1)v_{(m+1)/p}]$ of the basis $[\cB^\circ((m+1)/p)_{<k/p}]$
 to the element $[x(-n_r)\cdots x(-n_1)x(-m/p-1)v_{m/p}]$ 
and 
$\bar\pi$ sends $[x(-n_r)\cdots x(-n_1)v_{m/p}]$ in $[\cB^\circ(m/p)_{<k/p}]$
 to $[x(-n_r)\cdots x(-n_r)v_{m/p+1}]$.
\end{proposition}

\begin{proof}
Let us write $\ef {k/p}{m/p}= \Mcirc{m/p}/J_{k,m}$ with
the $F(p)$-submodule $J_{k,m}=\sum_{n\geq 0}x(-(k+n)/p)\Mcirc{m/p}\subset \Mcirc{m/p}$.
Since $\iota$ in \eqref{efrrpre} is an $F(1/p)$-module homomorphism,
we have 
$\iota \circ x(-(k+n)/p)=x(-(k+n)/p)\circ \iota $ for all $n\geq 0$
and the same holds for $\pi$.
It then follows that $\iota(J_{k,m+1})\subset J_{k,m}$ and $\pi(J_{k,m})\subset J_{k,m+p}$.
Therefore, the complex of the form \eqref{efrr} of $F(p)$-modules is induced.
By comparing bases, we see that it is exact, which completes the proof.
\end{proof}

Since the morphisms in \eqref{efrr} respects the bigradings and 
different summands in \eqref{efdecomp} have no common charges, 
we have the following.

\begin{corollary}\label{corefrr}
If $k> m+p$, then 
there is an exact sequence of bigraded $F(p;\chb,\cwb)$-modules
\begin{align*}
&0\ra \ef {k/p}{\frac{m+1}p,i;\chm+\chb^\circ,E+\frac mp+\cwb^\circ}
\ra \ef  {k/p}{\frac mp,i;\chm,\cwm}
\ra \ef  {k/p}{\frac{m}p+1,i;\chm,\cwm}
 \ra 0.
\end{align*}
for any $i=0,1,\ldots,p-1$.
In particular, we have an exact sequence
\begin{align}
&0\ra \rf {k/p}{\frac{m+1}p;\chm+\chb^\circ,E+\frac mp+\cwb^\circ}
\ra \rf  {k/p}{\frac mp;\chm,\cwm}
\ra \rf  {k/p}{\frac{m}p+1;\chm,\cwm}
 \ra 0.\label{efrrsummand}
\end{align}
\end{corollary}

\section{Switching isomorphism}\label{secswitch}
In this section, we study  RF- and EF-type modules and $\fd km$ 
by using lattice realization.

\subsection{Preliminaries}\label{subsecprepq}
Let $p$ be a positive integer.
Let $\alpha$ be an element of $\ah$ such that $(\alpha,\alpha)=p$
and set $Q=\ZZ\alpha$.
Consider the dual element $\varpi=\alpha/p$ with the dual lattice $P=\ZZ\varpi$ of $Q$.
We have the lattice vertex algebra $V_Q$ and lattice GVA $V_P$ with the  lattice principal subspaces $W_Q=\langle e^\alpha\rangle$ and $W_P=\langle e^\varpi\rangle$.
We have the embedding $V_Q\subset V_P$ and $W_Q\subset W_P^{(Q)}\subset W_P$.
Here and after, for any subset $A\subset P$ and any graded subspace $M=\bigoplus_{\lam\in P}M^\lam$ of $V_P$ with $M^\lam=M\cap V_P^\lam$,
we write $M^{(A)}=\bigoplus_{\lam\in A}M^\lam$.

We now consider submodules of $V_P$.
We have the decomposition
$
V_P=\bigoplus_{i=0}^{p-1} V_{Q+i\varpi}
$
as $V_Q$-modules, hence as $W_P^{(Q)}$-modules (and therefore as $W_Q$-modules).
Moreover, by considering charges, we have
\begin{equation}\label{eqnwpdecomp}
W_P(-k\varpi)=\bigoplus_{i=0}^{p-1} W_P(-k\varpi)^{(Q+i\varpi)},
\quad
W_P(-k\varpi)^\vee=\bigoplus_{i=0}^{p-1} W_P(-k\varpi)^{\vee,(i\varpi+Q)},
\end{equation}
as $W_P^{(Q)}$-modules, and hence as $W_Q$-modules.
Recall that
$I_{k}(V_P)=I_{k,e^\alp}(V_P)=W_P(-k\varpi)$
and
$C_{k-p+2}(V_P)\cong W_P(-k\varpi)^\vee$ as $W_P$-modules
(\Cref{propinv}, \Cref{propinvdual}).
By \eqref{eqnwpdecomp}, we have
\begin{align}
&I_k(V_P)=\bigoplus_{i=0}^{p-1}I_k(V_P)^{(Q+i\varpi)},
\quad I_k(V_P)^{(Q+i\varpi)}=I_k(V_{Q+i\varpi})=W_P(-k\varpi)^{(Q+i\varpi)},\label{eqn:ikvpdecomp}
\\
&C_{k-p+2}(V_P)=\bigoplus_{i=0}^{p-1}C_{k-p+2}(V_P)^{(Q+i\varpi)},
\quad C_{k-p+2}(V_P)^{(Q+i\varpi)}=C_{k-p+2}(V_{Q+i\varpi})
\cong W_P(-k\varpi)^{\vee,(i\varpi+Q)},\label{eqn:calpqp}
\end{align}
as  $W_P^{(Q)}$-modules and hence as $W_Q$-modules.

The algebras $F(p)$ and $F(1/p)$  generated by $a$ and $x$
are identified with $W_Q$ and $W_P$ via $a\mapsto e^\alpha$
and $x\mapsto e^\varpi$.
Then we have the isomorphisms $W_Q(m\varpi)\cong \M m$ and $W_P(m\varpi)\cong \Mcirc{m/p}$.

\subsection{Embedding into the ambient quotient space}
\label{subsecemb}

Recall from \Cref{secfib} that $\fd k m\cong C_{k}(\M m)$ as $F(p)$-modules.
By using the result of \Cref{subsecRRprin}
and \eqref{fibrr}, we induce an isomorphism of exact sequences
\begin{equation}\label{eqnoutside}
  \xymatrix{
    0\ar[r]& \fd k{m+p} \ar[r] \ar[d]_{C_k(\chi_{m+p})} & \fd km\ar[r] \ar[d]_{C_k(\chi_m)}&\fd k{m+1}\ar[r]\ar[d]_{C_k(\chi_{m+1})}&0\\
   0\ar[r] &C_k(W_Q((m+p)\varpi))\ar[r]^{C_k(\subset)} &C_k(W_Q(m\varpi)) \ar[r]^{C_k(\varphi)} &C_k(W_Q((m+1)\varpi))\ar[r]& 0,
  }
\end{equation}
where $\varphi=e^{\beta^\circ}(-m/p-1)$.

We now show that it induces a RR exact sequence inside $C_k(V_P)$. 
Consider the embedding $W_Q(m\varpi)\subset V_{Q+m\varpi}$ and then we have a morphism 
$C_k(W_Q(m\varpi))\ra C_k(V_{Q+m\varpi})$.
The image of the morphism is $[W_Q(m\varpi)]=W_Q(m\varpi)+J_k(V_{Q+m\varpi})$
with $J_k(V_{Q+m\varpi})=\sum_{n\geq 0}e^\alp(-k-n)V_{Q+m\varpi}$.
Let us denote the restriction of the morphism to the image by $\psi_m:C_k(W_Q(m\varpi))\twoheadrightarrow [W_Q(m\varpi)]$.
The morphisms $\psi_m$ induce a morphism of complexes
\begin{equation}\label{eqninside}
  \xymatrix{
    0\ar[r]& C_k(W_Q((m+p)\varpi)) \ar[r] \ar[d]_{\psi_{m+p}} & C_k(W_Q(m\varpi))\ar[r] \ar[d]_{\psi_m}&C_k(W_Q((m+1)\varpi))\ar[r]\ar[d]_{\psi_{m+1}}&0\\
   0\ar[r] &[W_Q((m+p)\varpi)]
   \ar@{^{(}->}[r]
   &[W_Q(m\varpi)] \ar[r]^{[C_k(\varphi)]} &[W_Q((m+1)\varpi)]\ar[r]& 0.
  }
\end{equation}
We show that $\psi_m$ ($m\in\ZZ$) are isomorphisms so that the lower complex in \eqref{eqninside} is exact.

By the injection of the lower exact sequence in \eqref{eqnoutside}, we have the sequence of natural embeddings
$$
C_k(W_Q(m\varpi))\subset C_k(W_Q((m-p)\varpi))\subset C_k(W_Q((m-2p)\varpi))\subset\cdots.
$$
Therefore, the direct limit $X=\bigcup_{n\geq0} 
C_k(W_Q((m-pn)\varpi))$ makes sense.
Recall that $\bigcup_{n\geq0} W_Q((m-pn)\varpi)=V_{Q+m\varpi}$, so that
$\bigcup_{n\geq0} [W_Q((m-pn)\varpi)]=C_k(V_{Q+m\varpi})$.
By taking the limit of $\{\psi_{m-pn}\}_n$, we have a surjective homomorphism
$\psi:X\ra C_k(V_{Q+m\varpi})$.
By comparing the characters of both sides, we see that $\psi$ is an injection and therefore, $\psi_m$ is also injective.
This proves that $\psi_m$ is an isomorphism.
Moreover, it follows by \eqref{eqn:calpqp} that the lower complex in \eqref{eqninside} is realized inside $C_k(V_P)$.
In summary, we have the following.

\begin{proposition}\label{propemb}
The quotient module $C_k(W_Q(m\varpi))$ is naturally embedded in the module $C_k(V_{Q+m\varpi})$ as  $[W_Q(m\varpi)]$.
The union $\bigcup_{n\geq 0}C_k(W_Q((m-pn)\varpi))$ coincides with $C_k(V_{Q+m\varpi})$.
Moreover, the lower complex in \eqref{eqninside} is an exact
sequence inside $C_k(V_P)$.
\end{proposition}

\subsection{Commutativity of invariance and coinvariance construction}

Let $k$ and $m$ be integers.

\begin{theorem}\label{thmcomm}
There is an isomorphism
$$
I_{e^\varpi,-m/p}(C_{e^\alp,k}(V_{m\varpi+Q}))
\cong C_{e^\alp,k}(I_{e^\varpi,-m/p}(V_{m\varpi+Q}))
$$
of $W_Q$-modules.
\end{theorem}

\begin{proof}
The right-hand side is $C_{e^\alp,k}(W_Q(m\varpi))$ by \Cref{propinv}.
By \Cref{propemb}, we have the natural embedding $C_{e^\alp,k}(W_Q(m\varpi))\subset C_{e^\alp,k}(V_{m\varpi+Q})$.
Let us set the left-hand side as $V=I_{e^\varpi,-m/p}(C_{e^\alp,k}(V_{m\varpi+Q}))$.
We clearly have $V\supset C_{e^\alp,k}(W_Q(m\varpi))$ inside $C_{e^\alp,k}(V_{m\varpi+Q})$. It remains to show the opposite inclusion.
Let $v$ be an element of $V$. Then we have $\ell\in\ZZ$ 
such that $v\in C_{e^\alp,k}(W_Q(\ell\varpi))$ by \Cref{propemb} again.
We show $v\in C_{e^\alp,k}(W_Q(m\varpi))$ by induction on $\ell$.
If $\ell\geq m$, we have nothing to prove.
Let us  suppose $\ell<m$. 
In this case, we have $e^\varpi(-\ell/p-1)v=0$ since 
$-\ell/p-1\geqc0 -m/p$.
By \eqref{eqnoutside}, we have $v\in C_{e^\alp,k}(W_Q((\ell+p)\varpi))$, which completes
the proof.
\end{proof}

The left-hand side of \Cref{thmcomm} has the form:
\begin{align*}
I_{e^\varpi,-m/p}(C_{e^\alp,k}(V_{m\varpi+Q}))&\cong
(C_{e^\varpi,-m/p-1/p+2}(I_{e^\alp,k+p-2}(V_{-m\varpi+Q})))^\vee\\
&=(C_{e^\varpi,-(m+1)/p+2}(W_P(-(k+p-2)\varpi)^{(Q-m\varpi)}))^\vee,
\end{align*}
where the first and second equalities follow from 
\eqref{dpcdp} and \eqref{eqn:ikvpdecomp}, respectively.

\begin{corollary} \label{corcomm}
For any integers $m$ and $k$, we have
$$
C_{e^\varpi,-(m+1)/p+2}\bigl(W_P(-(k+p-2)\varpi)^{(Q-m\varpi)}\bigr)
\cong C_{e^\alp,k}(W_Q(m\varpi))^\vee.
$$
\end{corollary}

\subsection{The switching isomophism}\label{subsecswitch}
Recall the quotient $F(p)$-module $\rf{k/p}{m/p}$.
We evidently have
$$
\rf{k/p}{\frac mp}\cong C_{e^\varpi, k/p}(W_P(m\varpi))^{((k+1)\varpi+Q)}.
$$

Then \Cref{corcomm} with $n=-m$ implies the following corollary.
Let us consider the normalized bigradings:
$F(p)=F(p;1,p/2)$ and $F(1/p)=F(1/p;1/p,1/2p)$.

\begin{corollary}\label{corswitch}
For any integers $m$ and $k$ and complex numbers $\chm$ and $\cwm$, there is an isomorphism
\begin{equation*}
\fd k{m;\chm+\frac mp,\cwm+\frac{m^2}{2p}}^\vee\cong 
\rf{-m'/p}{-\frac{k'}p;-\chm-\frac{k'}p,\cwm+\frac{k'^2}{2p}}
\end{equation*}
of bigraded $F(p;1,p/2)$-modules, where 
$m'=m-2p+1$ and $k'=k+p-2$.
\end{corollary}

In particular, $\rf{k/p}{m/p}$ is cogenerated by a single cogenerator.

We now derive Fibonacci exact sequence \eqref{fibrr2}  from a RR exact sequence for EF-type modules \Cref{corefrr}.
Note that \Cref{corefrr} is induced from \eqref{ourintro-} with $g=1/p$.
Let $m,k,m',k'$ be as above.
By using \Cref{corefrr}, we have the exact sequence
\begin{align*}
&0\ra \rf{-m'/p}{\frac{-k'+1}p;\frac{-k'+1}p,\frac{(k'-1)^2}{2p}}
\ra \rf{-m'/p}{-\frac{k'}p;-\frac{k'}p,\frac{k'^2}{2p}}\\
&\qquad\ra \rf{-m'/p}{\frac{-k'}p+1;-\frac{k'}p,\frac{k'^2}{2p}}\ra0
\end{align*}
of bigraded $F(p;1,p^2/2)$-modules.
By \Cref{corswitch}, it is equivalent to the exact sequence
$$
0\ra \fd {k-1}{m;\frac mp, \frac{m^2}{2p}}^\vee
\ra \fd k{m;\frac mp, \frac{m^2}{2p}}^\vee
\ra \fd{k-p}{m;\frac mp+1, \frac{m^2}{2p}+k+\frac p2-2}^\vee\ra0.
$$
Since $\vee$ is a contravariant functor, by applying $\vee$ to 
it, we have the exact sequence
\begin{equation}\label{fibrr2p}
0\ra \fd{k-p}{m;\frac mp+1, \frac{m^2}{2p}+k+\frac p2-2}
\xrightarrow {f_1} \fd k{m;\frac mp, \frac{m^2}{2p}}
\xrightarrow{f_2} \fd {k-1}{m;\frac mp, \frac{m^2}{2p}}
\ra0.
\end{equation}
We now show that it is equivalent to \eqref{fibrr2}.
It suffices to show that $f_1(v_{k-p,m})$ is a non-zero multiple
of $a(-m-1)v_{k,m}$ and $f_2(v_{k,m})$ a non-zero multiple of $v_{k-1,m}$.
We may track the invariance and coinvariance construction 
in \Cref{thmcomm} to show this in principle.
We here employ the charge grading in \eqref{fibrr2p}.
Note that the generator $v_{n,\ell}$ of  $\fd n\ell$ ($n,\ell\in\ZZ$)
is a unique lowest charge vector of $\fd n\ell$ up to scalar multiples (cf.~\Cref{subsecfreefg}).
Let us write \eqref{fibrr2p} as $0\ra M_1\xrightarrow{f_1} M_2 \xrightarrow{f_2} M_3\ra 0$.
Since the  lowest charges of $M_2$ and $M_3$ coincide with each other,
the surjection $f_2$ sends $v_{k,m}$ to a scalar multiple of $v_{k-1,m}$.
Therefore,  $f_2$ coincides with a scalar multiple of the surjection in  \eqref{fibrr2}.
It then follows that $\ker{f_2}$ coincides with the image of the injection in \eqref{fibrr2}.
Hence, $a(-m-1)v_{k,m}\in M_2$ is the unique lowest charge vector of $\ker{f_2}$ up to constant multiples.
Since $\ker{f_2}=\im(f_1)$ by the exactness, by using the same argument, we see that $f_1(v_{k-p,m})$ is  a scalar multiple of $a(-m-1)v_{k,m}$, which completes the proof.

Thus, we see that \eqref{fibrr2} is induced from \eqref{ourintro-} with $g=1/p$.

Note that we can obtain an exact sequence for RF-type modules by applying the restricted dual functor to \eqref{fibrr} and it should be called a Fibonacci type exact sequence for RF-type modules.

%

\subsection{Character identity}
Let $m$ and $k$ be integers. 
Let us consider the character identity 
\begin{equation}\label{eqndualchar}
\ch\left[\fd k{m;\chm+\frac mp,\cwm+\frac{m^2}{2p}}\right](z^{-1},q)
= 
\ch\left[\rf{-m'/p}{-\frac{k'}p;-\chm-\frac{k'}p,\cwm+\frac{k'^2}{2p}}\right](z,q),
\end{equation}
which follows from \Cref{corswitch}.
We can compute the characters of RF-type modules by using \Cref{thmfibcirc} and they are described in \Cref{appendixchar}.
By dividing out the common factors $z^Sq^E$ from the characters, we have the following.

\begin{corollary}\label{corchar}
Let $p$ be a positive integer and $m,k$ integers.
Let $i$ be the remainder of $-m+k-2$ divided by $p$.
We then have
\begin{align*}
&z^{-m/p}q^{m^2/2p}
\sum_{r\geq 0}
q^{pr^2/2+mr}
\begin{bmatrix}
k-m+p-2+(1-p)r\\
r
\end{bmatrix}_q
z^{-r}
\\
&\quad =
z^{(-k'+i)/p}q^{(k'^2+i^2-2k'i)/2p}
\sum_{r\geq0}
q^{pr^2/2+(i-k')r}
\begin{bmatrix}
\frac{-m'+k'-i+1}p -2+i+(p-1)r
\\ i+pr
\end{bmatrix}_q
z^{r},
\end{align*}
 where
$m'=m-2p+1$ and $k'=k+p-2$.
\end{corollary}

For example let $p=2$, $k=0$ and $m=-2n$ with $n\geq 0$.
In this case, \Cref{corcomm} implies
$$
C_{e^\varpi,n+3/2} (W_P^{(Q)})
\cong C_{e^\alp,0}(W_Q(-n\alpha))^\vee.
$$
The corresponding character identity reads
$$
z^{n}q^{n^2}\sum_{r=0}^{n}q^{r^2-2nr}
\begin{bmatrix}
2n-r\\r
\end{bmatrix}_q
z^{-r}
=
\sum_{r=0}^{n}q^{r^2}\begin{bmatrix}
n+r\\2r\end{bmatrix}_q z^{r}.
$$

\section{BFL character decomposition formula}\label{seccomp}

In this section, we consider relations between our RR exact sequences and the character decomposition formula appearing in \cite{BFL} and \cite{Ke2}.

\subsection{Finite flags in free modules}
Let $p$ be a positive integer, $\chb$ a non-zero complex number
and $\cwb$ a complex number.
Let $n$ be a non-negative integer and consider a free module 
$\M {-n;\chm,\cwm}$ 
over $F(p;\chb,\cwb)$.
By applying \Cref{prop:fund} recursively, we have a flag in $\M {-n;\chm,\cwm}$ as follows.

\begin{theorem}\label{thm:flag}
For any $n\geq 0$, the RR exact sequences induce  a finite flag
in the module $\M{-n}$
where each factor is isomorphic to one of $\M 0,\M1,\ldots,\M{p-1}$.

Moreover, there is a linear isomorphism 
\begin{align*}
\M{-n;\chm,\cwm}&\cong 
\fd{-p+2}{-n;\chm,\cwm}\otimes_\CC \M{0;0,0}\\
&\quad \oplus \bigoplus_{i=1}^{p-1}
\fd{-2p+i+2}{-n;\chm+\chb,\cwm-\cwb+i}\otimes_\CC \M{i;0,0}
\end{align*}
of bigraded vector spaces.
\end{theorem}

It is easy to see as follows.
We display the exact sequence in \Cref{prop:fund} as
$$
\xymatrix{
& \ar[ld]_{a(-m-1)} \M{m;\chm,\cwm} \ar[d]^=&\\
\M{m+p;\chm+\chb,\cwm+m+\cwb}&\M{m+1;\chm,\cwm}
}
$$
Here, the left down arrow is the embedding towards the opposite direction
and the vertical down arrow is the quotient.
The labels on the paths indicate how the combinatorial bases of the nodes are related.
Say, the label 
``$a(-m-1)$'' indicates that a vector $a(-n_r)\cdots a(-n_1)v_{m+p}$ of the basis $\cB(m+p)$ of the leaf $\M{m+p}$ maps to
the vector $a(-m-1)a(-n_r)\cdots a(-n_1)v_m$ of the basis $\cB(m)$ of the root $\M m$.
The label ``$=$'' indicates that a vector $a(-n_r)\cdots a(-n_1)v_m$
maps to $a(-n_r)\cdots a(-n_1)v_{m+1}$.

We connect such diagrams to make a tree such that 
 the root is $\M {-n}$ and all leaves are isomorphic to
one of $\M 0,\M 1,\ldots,\M{p-1}$.
For example, set $p=2$, $\chb=1$, $\cwb=1$ and write $e=a$. 
We then have the following trees
for $n=1,2,3$.

\medskip
$
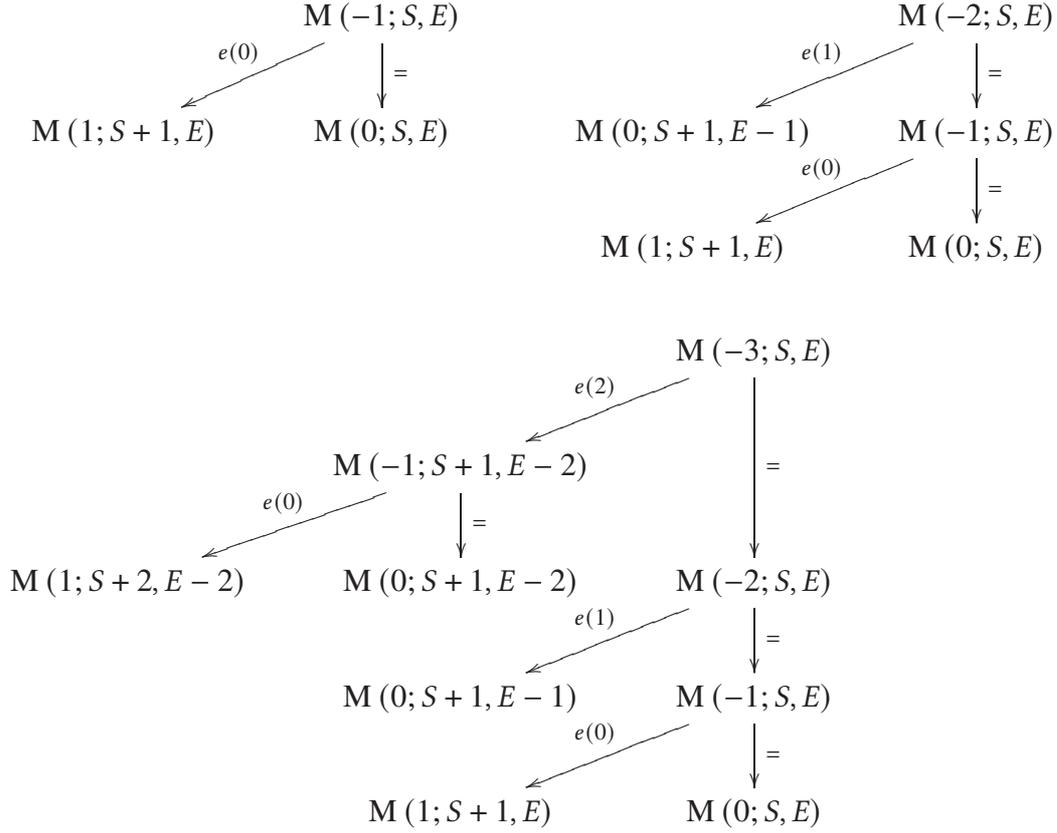
\xymatrix{
& \ar[ld]_{e(0)} \M{-1;\chm,\cwm} \ar[d]^=&\\
\M{1;\chm+1,\cwm}&\M{0;\chm,\cwm}
}
$
$
\xymatrix{
& \ar[ld]_{e(1)} \M{-2;\chm,\cwm} \ar[d]^=&&\\
\M{0;\chm+1,\cwm-1}&\ar[ld]_{e(0)} \M{-1;\chm,\cwm} \ar[d]^=&\\
\M{1;\chm+1,\cwm}&\M{0;\chm,\cwm}
}
$

\medskip
$$
\xymatrix{
&& \ar[ld]_{e(2)} \M{-3;\chm,\cwm} \ar[dd]^=&&\\
& \ar[ld]_{e(0)} \M{-1;\chm+1,\cwm-2} \ar[d]^=&&&&&\\
\M{1;\chm+2,\cwm-2}&\M{0;\chm+1,\cwm-2}&\ar[ld]_{e(1)} 
\M{-2;\chm,\cwm} \ar[d]^=&\\
&\M{0;\chm+1,\cwm-1}&\ar[ld]_{e(0)} \M{-1;\chm,\cwm} \ar[d]^=&\\
&\M{1;\chm+1,\cwm}&\M{0;\chm,\cwm}
} 
$$

\medskip
We can find a flag in $\M {-n}$ by looking at the tree.
For example, consider the case of $n=2$ from the above trees.
We first set $F_1=\M{0;\chm+1,\cwm-1}\subset \M{-2;\chm,\cwm}$.
We have the canonical projection $\pi: \M{-2;\chm,\cwm}\ra \M{-1;\chm,\cwm}$.
We then set $F_2=\pi^{-1}(\M{1;\chm+1,\cwm})\subset \M{-2;\chm,\cwm}$.
We now have a flag
$
0=F_0\subset F_1\subset F_2\subset F_3=\M{-2;\chm,\cwm}
$
of $F(2;1,1)$-modules such that $F_1/F_0\cong \M 0$,
$F_2/F_1\cong \M1$ and $F_3/F_2\cong \M0$.

Let us consider such a flag $0=F_0\subset F_1\subset \cdots \subset F_r=\M {-n}$.
Let $\mathrm{gr}\,\M{-n;\chm,\cwm}$ denote the graded module of $\M{-n}$ with respect to the flag.
It decomposes into the sum
$$
\mathrm{gr}\, \M{-n;\chm,\cwm}\cong \bigoplus_{i=0}^{p-1}V(i)\otimes_\CC \M{i;0,0}
$$
of bigraded $F(p;\chb,\cwb)$-modules with
  multiplicity spaces $V(0),\ldots,V(p-1)$ equipped with bigradings.

Let $i$ be an element of $\{0,\ldots,p-1\}$. 
A {\em standard branching vector} of a leaf $L\cong \M i$
 is an element of $\cB(-n)$ which corresponds to the generator
 $v_i$ of $L$ if we track the paths in the tree as indicated in the labels.
Let $B_i$ be the set  of standard branching vectors  of all leaves isomorphic to $\M i$.
The linear span of $B_i$ is isomorphic to $V(i)$
as bigraded vector spaces.

For example, consider the case of $n=3$ from the above trees.
Then we have 3 and 2 leaves isomorphic to $\M0$ and  $\M1$, 
respectively, and we have
$$
B_0=\{v_{-3},  e(1)v_{-3}, e(2)v_{-3}\},\qquad 
B_1=\{e(0)v_{-3}, e(0)e(2)v_{-3}\}.
$$
For any $k,m\in\ZZ$, let us write $\cB(m)_{<k}$ the subset of $\cB(m)$ which consists of monomials \eqref{basisfm} satisfying $n_r<k$.
We see that $B_0=\cB(-3)_{<0}$ and $B_1=e(0)\cB(-3)_{<-1}$,
whose images coincide with the bases $\cB(0,-3)$ and $\cB(-1,-3)$ of $\fd 0{-3}$ and $\fd {-1}{-3}$ introduced in \Cref{subsecfreefda}, respectively.
It implies that $V(0)\cong \fd0{-3;S,E}$ and $V(1)\cong \fd{-1}{-3;\chm+1,\cwm}$ as bigraded vector spaces.
As a result, we have a bigraded linear isomorphism
$$
\M{-3;S,E}\cong \fd0{-3;S,E}\otimes_\CC \M{0;0,0}\oplus 
\fd{-1}{-3;\chm+1,\cwm}\otimes_\CC \M{1;0,0}.
$$

This is explained by using the basis $\cB(-n)$ of $\M{-n}$.
For example, if $p=2$, we have the partition
\begin{align}\label{eqnpart2}
\cB(-n)&=\bigsqcup_{n_1\geq1, n_{i+1}-n_i\geq 2}
e(-n_r)\cdots e(-n_1)\cB(-n)_{<0}
\sqcup\bigsqcup_{n_1\geq2, n_{i+1}-n_i\geq 2}
e(-n_r)\cdots e(-n_1)e(0)\cB(-n)_{<-1}
\end{align}
of $\cB(-n)$.
The sets $\cB(-n)_{<0}$ and $\cB(-n)_{<-1}$ 
bijectively map to 
the bases of $\fd{0}{-n}$ and $\fd{-1}{-n}$.
Moreover, if we apply the set of operators $
\{e(-n_r)\cdots e(-n_1)\,|\,n_1\geq1, n_{i+1}-n_i\geq 2\}$ in the first summand of \eqref{eqnpart2} to $v_0\in \M0$, we have the basis 
$\cB(0)$ of $\M0$.
Similarly, the basis $\cB(1)$ of $\M1$ coincides with 
$
\{e(-n_r)\cdots e(-n_1)\,|\,n_1\geq2, n_{i+1}-n_i\geq 2\}$ in the second summand of \eqref{eqnpart2} if applied to $v_1\in \M1$.

It readily generalizes to the case of any $p>0$ and we have \Cref{thm:flag}.

Note that \Cref{thm:flag}  gives a character decomposition of $\M{-n;\chm,\cwm}$
with respect to the characters of $\M 0,\ldots,\M{p-1}$
and $\fd km$.
Note that we may replace $\fd km$ with $\rf{-m'/p}{-k'/p}$
by using the switching isomorphism \eqref{eqndualchar}.

\subsection{Flags in lattice vertex algebra modules}\label{subsecchar}

Consider the algebras $V_Q$, $W_Q$, $W_P$ and their modules.
Let $n$ be a non-negative integer.
Recall that $W_Q(-n\varpi)\cong \M{-n;-n/p,n^2/2p}$ via the isomorphism
$W_Q\cong F(p;1,p/2)$.
It follows by \Cref{thm:flag} that the RR exact sequences
for lattice principal subspaces induce a finite flag 
\begin{equation}\label{flagforwq}
0=F_0\subset F_1\subset \cdots \subset F_r=W_Q(-n\varpi)
\end{equation}
such that each factor module $F_j/F_{j-1}$ is isomorphic to one of 
$W_Q(0), W_Q(\varpi),\ldots,W_Q((p-1)\varpi)$.
It then follows that there is a linear isomorphism
\begin{align}\label{eqn:wqfinchardecomp}
W_Q(-n\varpi)&\cong 
C_{e^\alp,-p+2}(W_Q(-n\varpi))\otimes_\CC W_Q(0) \nonumber\\
&\quad \oplus \bigoplus_{i=1}^{p-1}
\CC_{1-i/p,-(p-i)^2/2p}\otimes_\CC 
C_{e^\alp,-2p+i+2}(W_Q(-n\varpi))\otimes_\CC W_Q(i\varpi)
\end{align}
of bigraded vector spaces,
where $\CC_{\chm,\cwm}$ is a one-dimensional vector space having charge $\chm$ and weight $\cwm$.

Let $\ell$ be one of $\{0,1,\ldots,p-1\}$ and consider the direct limit $V_{Q-\ell\varpi}\cong\varinjlim W_Q(-(\ell+pi)\varpi)$ for  the series of
embedding
\begin{equation}\label{eqnemb}
W_Q(-\ell\varpi)\subset W_Q(-(\ell+i)\varpi)\subset W_Q(-(\ell+2i)\varpi)\subset \cdots.
\end{equation}
Since  flag \eqref{flagforwq}
factors through \eqref{eqnemb},
we have an infinite flag
\begin{equation}\label{infiniteflag}
0=F_0\subset F_1\subset F_2\subset \cdots (\subset V_{Q-\ell\varpi})
\end{equation}
such that each factor module $F_{j+1}/F_j$ is isomorphic to
one of $W_Q(0), W_Q(\varpi),\ldots,W_Q((p-1)\varpi)$.
We have, by using \eqref{eqn:wqfinchardecomp} with $n=\ell+pi$ with $i\ra\infty$,
 a bigraded linear isomorphism
\begin{align*}
V_{Q-\ell\varpi}&\cong 
C_{e^\alpha,-p+2}(V_{Q-\ell\varpi})\otimes_\CC W_Q(0)
\oplus \bigoplus_{i=1}^{p-1}
\CC_{1-i/p,-(p-i)^2/2p}\otimes_\CC C_{e^\alpha,-2p+i+2}(V_{Q-\ell\varpi})\otimes_\CC W_Q(i\varpi).
\end{align*}
Let $k$ be an integer. 
Recall from \eqref{eqn:calpqp} that 
$$
C_{e^\alp,k-p+2}(V_{Q-\ell\varpi})\cong W_P(-k\varpi)^{\vee,(Q-\ell\varpi)}
\cong W_P(-k\varpi)^{(Q+\ell\varpi),\vee}.
$$
Hence, we have the character identity
\begin{align*}
\ch[C_k(V_{Q-\ell\varpi})](z,q)&=\ch[W_P(-k\varpi)^{(Q+\ell\varpi)}](z^{-1},q)
=\sum_{r\geq0
\ \mathrm{s.t.}\ pr+k+\ell\geq 0} \frac{q^{(pr+\ell)^2/2p}}
{(q)_{pr+k+\ell}}z^{-r-\ell/p}.
\end{align*}
Now, we have obtained a character decomposition of $V_{Q-\ell\varpi}$ given in \cite{BFL} for $p=2$ and recent preprint \cite{Ke2}:
\begin{align}\label{eqnmainchar}
\ch[V_{Q-\ell\varpi}](z,q)&=\sum_{r=0}^\infty \frac{q^{(pr+\ell)^2/2p}}
{(q)_{pr+\ell}}z^{-r-\ell/p}
\sum_{s=0}^\infty \frac{q^{ps^2/2}}{(q)_s}z^s \nonumber\\
&\quad +z\sum_{i=1}^{p-1} q^{-p/2+i}
\sum_{r\geq0
\ \mathrm{s.t.}\ pr-p+i+\ell\geq 0} \frac{q^{(pr+\ell)^2/2p}}
{(q)_{pr-p+i+\ell}}z^{-r-\ell/p}
\sum_{s=0}^\infty \frac{q^{ps^2/2+si}}{(q)_s}z^s.
\end{align}

We finally remark the following.
Recall that $C_{a,0}(\M m)\cong \fd 0m$.
By applying the functor $C_{a,0}(\cdot)$ to the flag in
\Cref{thm:flag}, we have a flag in $\fd 0{-n}$ with factor modules 
$\fd 00,\fd01,\ldots, \fd0{p-1}$ and it is in fact a composition series of $\fd0{-n}$ since the factor modules are all one-dimensional.
Similarly, if we apply the functor $C_{e^\alp,0}(\cdot)$ to the infinite flag \eqref{infiniteflag}, we  have a composition series of $C_{e^\alp,0}(V_{Q-\ell\varpi})\cong W_P(-(p-2)\varpi)^{(Q+\ell\varpi),\vee}$.

\appendix

\section{Proofs}
\label{secfdproofs}

In this appendix, we write proofs of several propositions.

\subsection{Bases of $\fda kp$ and $\fd km$}\label{subseccoinvva}
In this section, we show that $\cB_k$ and $\cB(k,m)$ are basis
 of $\fda kp$ and  $\fd km$, by using the combinatorial bases of 
 $F(p)$ and $\M m$.

Recall that the elements
\begin{equation}\label{eqn:basisatype}
a(-n_r)\cdots a(-n_1)\bm1\quad (r\geq0, 
n_r\geqc p n_{r-1}\geqc p \cdots\geqc p n_1\geqc 0 1)
\end{equation}
form a combinatorial basis $\cB$ of $F(p)$.
Consider the partition $\cB=\cB_{<k}\sqcup \cB_{\geq k}$ of $\cB$, where $\cB_{<k}$ and $\cB_{\geq k}$ are the subsets of $\cB$ consisting of all basis elements \eqref{eqn:basisatype} satisfying $n_r< k$ and $n_r\geq k$, respectively.
Note that the image of $\cB_{<k}$ coincides with the basis $\cB_k$
of $\fda kp$.

\begin{theorem}\label{thmfib}
The images of the elements of $\cB_{<k}$  form a $\CC$-basis 
of $\fda kp$.
\end{theorem}

\begin{proof}
Since we have the partition $\cB=\cB_{<k}\sqcup \cB_{\geq k}$,
it suffices to show that $J_k(F(p))=\cspan(\cB_{\geq k})$.
The inclusion $(\supset)$ is clear. We show $(\subset)$.
Let  $v$ be an element of $J_k(F(p))$.
It suffices to show $v\in\cspan (\cB_{\geq k})$ when $v$ is a monomial of the form $a(-m_r)\cdots a(-m_1)\bm1$
 with $r\geq 1$, $m_1,\ldots, m_r\in\ZZ$ such that
 $m_r\geq k$. 
 We show $v\in \cspan(\cB_{\geq k})$ by induction on $n=m_1+\cdots+m_{r-1}$.
 If $n$ is small enough, then we have $v=0\in\cspan(\cB_{\geq k})$.
 Now, let us consider a general $n$.
By using the basis $\cB$, we see that $v$ is a sum of monomials of the form
$w=a(-m_r)a(-n_{r-1})\cdots a(-n_1)\bm1$ such that
$n_1,\ldots,n_{r-1}$ satisfy 
$n_{r-1}\geqc p \cdots\geqc p n_1\geqc 0 1$.
 It suffices to show $w\in\cspan(\cB_k)$.
 We divide the proof into two cases: (i) $m_r-n_{r-1}\geq p$,
 (ii) $m_r-n_{r-1}< p$.
 If $m_r-n_{r-1}\geq p$, then we have $w\in \cB_{\geq k}$.
If $m_r-n_{r-1}< p$,  by using \eqref{lemfund} with $b=a$ and $g=p$ and the induction hypothesis, we have $w\in \cspan(\cB_{\geq k})$.
Thus, $w\in \cspan(\cB_{\geq k})$, as desired.
\end{proof}

Let $m$ be an integer. 
Consider the partition $\cB(m)=\cB(m)_{<k}\sqcup \cB(m)_{\geq k}$ of $\cB(m)$, where $\cB(m)_{<k}$ and $\cB(m)_{\geq k}$ are the subsets of $\cB(m)$ consisting of all monomials
\eqref{basisfm} satisfying $n_r< k$ and $n_r\geq k$, respectively.
The following theorem is proved in the same way as the proof of \Cref{thmfib}.
Note that the image of $\cB(m)_{<k}$ coincides with $\cB(k,m)$.

\begin{theorem}\label{thmfibm}
The images of the elements of $\cB(m)_{<k}$  form a $\CC$-basis $\cB(k,m)$ of $\fd km$.
\end{theorem}

\subsection{Proof of \Cref{thmfibcircvac}}\label{proofthmfibcirc}

We here write a proof of \Cref{thmfibcircvac}.

\begin{proof}[Proof of \Cref{thmfibcircvac}]
The proof is similar to that of \Cref{thmfib}.
Set $J_N=\sum_{n\in\Znn}x(-k/p-n/p)F(1/p)$, so that $\ef{k/p}0=F(1/p)/J_k$.
Since we have the partition $\cB^\circ=\cB^\circ_{<k/p}\sqcup \cB^\circ_{\geq k/p}$,
it suffices to show that $J_k=\cspan(\cB^\circ_{\geq k/p})$.
The inclusion $(\supset)$ is clear. We show $(\subset)$.
Let  $v$ be an element of $J_k$.
It suffices to show $v\in\cspan (\cB^\circ_{\geq k/p})$ when $v$ is a monomial of the form $x(-m_r)\cdots x(-m_1)\bm1$
 with $r\geq 1$, $m_1\in\Zp$, $m_{i+1}-m_i\in 1/p+\ZZ$ ($1\leq i\leq r-1$) such that
 $m_r\geq k/p$. 
 We show $v\in \cspan(\cB^\circ_{\geq k/p})$ by induction on $n=m_1+\cdots+m_{r-1}$.
 If $n$ is small enough, then we have $v=0\in\cspan(\cB^\circ_{\geq k/p})$.
 Now, let us consider a general $n$.
By using the basis $\cB^\circ$, we see that $v$ is a sum of monomials of the form
$w=x(-m_r)x(-n_{r-1})\cdots x(-n_1)\bm1$ such that
$n_1,\ldots,n_{r-1}$ satisfy 
$n_{r-1}\geqc {1/p} \cdots\geqc {1/p} n_1\geqc 0 1$.
 It suffices to show $w\in\cspan(\cB^\circ_{\geq k/p})$.
 We divide the proof into two cases: (i) $m_r-n_{r-1}\geq 1/p$,
 (ii) $m_r-n_{r-1}< 1/p$.
 If $m_r-n_{r-1}\geq 1/p$, then we have $w\in \cB^
 \circ_{\geq k/p}$.
If $m_r-n_{r-1}< 1/p$,  by using \eqref{lemfund} with $b=x$ and $g=1/p$ and the induction hypothesis, we have $w\in \cspan(\cB^\circ_{\geq k/p})$.
Thus, $w\in \cspan(\cB^\circ_{\geq k/p})$, as desired.
\end{proof}

\subsection{Character formulas for RF- and EF-type modules}
\label{appendixchar}

Let us compute the character of $\ef{k/p}0$.
By \Cref{thmfibcircvac}, the elements of $[\cB^\circ_{<k/p}]$ which belong to $\ef{k/p}0_{r,d}$  have the form
$$
\left[x\left(-m_r-\frac{r-1}p-1\right)x\left(-m_{r-1}-\frac{r-2}p-1\right)\cdots
 x\left(-m_2-\frac1p-1\right)x(-m_1-1)\bm1\right]
$$
with $m_1,\ldots,m_r\in\ZZ$ such that $d=m_1+\cdots+m_r+r+r(r-1)/2p$, $m_r\geq \cdots \geq m_1\geq 0$ and 
\begin{equation}\label{eqncondmr1}
m_r+\frac{r-1}p+1<\frac kp.
\end{equation}
Since $m_r\in\ZZ$, it follows that \eqref{eqncondmr1} is equivalent to $m_r<\lfloor (k-r+1)/p\rfloor-1$ and it is equivalent to $m_r\leq \lfloor (k-r)/p\rfloor-1$ as $k,r\in\ZZ$.
Therefore, the elements are parametrized by the partitions of the integer $d$ into at most $r$ parts and at most $\lfloor (k-r)/p\rfloor-1$. It implies that the character
 $\ch[\ef{k/p}0](z,q)$ of $\ef{k/p}{0;S,E}$ has the form
 $$
\ch\left[\ef{k/p}0\right](z,q)=\sum_{r\geq0}
q^{r^2/2p+(1-1/2p)r}\begin{bmatrix}
\lfloor \frac{k-r}p\rfloor-1+r \\r\end{bmatrix}_qz^{ r}.
$$
If we consider the bigraded module $\ef{k/p}{0;S,E}$ over $F(p;C,D)$, we have
$$
\ch\left[\ef{k/p}{0;S,E}\right](z,q)=z^Eq^S\sum_{r\geq0}
q^{r^2/2p+(\cwb^\circ-1/2p)r}\begin{bmatrix}
\lfloor \frac{k-r}p\rfloor-1+r \\r\end{bmatrix}_qz^{C^\circ r}.
$$
We now compute the character of $\rf {k/p}{0;S,E}\cong \ef{k/p}{0,0;S,E}$.
Let $i\in\{0,\ldots,p-1\}$ be the remainder of $k+1$ modulo $p$.
We then have
\begin{align*}
\ch\left[\rf{k/p}{0;S,E}\right](z,q)
&=
z^{S+C i/p}q^{E+i^2/2p+i(D^\circ-1/2p)}\\
&\quad\cdot\sum_{r\geq0}
q^{pr^2/2+(pD^\circ-1/2+i)r}
\begin{bmatrix}
\frac{k-i+1}p-2+i+(p-1)r \\i+pr
\end{bmatrix}_q
z^{Cr}.
\end{align*}

\Cref{thmfibcirc} implies that the character of $\ef{k/p}{m/p;\chm,\cwm}$ has the form
$$
\ch\left[\ef{k/p}{\frac mp;\chm,\cwm}\right](z,q)=q^Ez^\chm\sum_{r\geq0}
q^{r^2/2p+(D^\circ-1/2p+m/p)r}
\begin{bmatrix}
\lfloor \frac{k-r-m}p\rfloor-1+r \\r\end{bmatrix}_qz^{C^\circ r}.
$$

We now compute the character of $\rf{k/p}{m/p;\chm,\cwm}$.
Let $i\in\{0,\ldots,p-1\}$ be the remainder of $k-m+1$ divided by $p$.
Since the space $\Mcirc{m/p}^{(1/p+\ZZ)}$
is spanned by  monomials
$x(-n_{rp+i})x(-n_{rp+i-1})\cdots x(-n_1)v_{m/p}$ of length 
$rp+i$ with $r\geq0$,
we  have
\begin{align}\label{eqndnfine}
&\ch\left[\rf{k/p}{\frac mp;\chm,\cwm}\right](z,q) 
=q^{E+i^2/2p+i(\cwb^\circ-1/2p+m/p)}z^{\chm+\chb i/p}
\nonumber\\ &\qquad\qquad\qquad\cdot
\sum_{r\geq0}
q^{pr^2/2+(p\cwb^\circ-1/2+i+m)r}
\begin{bmatrix}
 \frac{k-m-i+1}p-2+i+(p-1)r \\i+pr
\end{bmatrix}_q
z^{Cr}.
\end{align}


\begin{thebibliography}{99999}

\bibitem[A]{A}
Andrews, G. ``The theory of partitions''.  Cambridge university press, (1998).

\bibitem[A04]{A04}
Andrews, G. ``Fibonacci numbers and the Rogers-Ramanujan identities.'' Fibonacci Quart. {\bf 42} (2004): 3--19.

\bibitem[BK]{BK} Bakalov, B., and V. Kac,  ``Generalized vertex algebras.'' Proceedings of the 6-th International Workshop ``Lie Theory and Its Applications in Physics'', Varna, Bulgaria 3--25 (2006).


\bibitem[BFL]{BFL}
Bershtein, M., B. Feigin, and A. Litvinov. ``Coupling of Two Conformal Field Theories and Nakajima-Yoshioka Blow-Up Equations.'' Lett.  Math. Phys. {\bf 106} (2016): 29--56.

\bibitem[B]{B}
Borcherds, R.  ``Vertex algebras, Kac-Moody algebras, and the Monster.'' Proc. Nat. Acad. Sci. {\bf 83} (1986): 3068--3071.


\bibitem[CalLM08]{CalLM08}
Calinescu, C., J. Lepowsky, and A. Milas, ``Vertex-algebraic structure of the principal subspaces of certain-modules, I: level one case.'' Int. J. Math. {\bf 19} (2008): 71--92.

\bibitem[CalLM14]{CalLM14}
Calinescu, C., J. Lepowsky, and A. Milas. ``Vertex-algebraic structure of principal subspaces of standard-modules, I.'' Int. J. Math. {\bf 25} (2014): 1450063.



\bibitem[CLM]{CLM03}
Capparelli, S., J. Lepowsky, and A. Milas. 
``The Rogers-Ramanujan recursion and intertwining operators.'' Comm. Contemp. Math. {\bf 5} (2003): 947--966.

\bibitem[Ca]{Ca}
Carlitz, L.
``Fibonacci notes. IV. $q$-Fibonacci polynomials.''
Fibonacci Quart. {\bf13} (1975): 97--102.

\bibitem[Ci]{Ci}
Cigler, J.
``$q$-Fibonacci polynomials and the Rogers-Ramanujan identities.''
Ann. Comb. {\bf 8} (2004): 269--285.

\bibitem[DL]{DL} Dong, C., and J. Lepowsky. ``Generalized vertex algebras and relative vertex operators.'' Springer, (1993).

\bibitem[DM]{DM}
Dotsenko, V., and S. Mozgovoy. ``DT invariants from vertex algebras.'' arXiv:2108.10338 (2021).


\bibitem[FF]{FF}
Feigin, B., and E. Feigin.
``Principal subspace for the bosonic vertex operator 
$\phi_{\sqrt{2m}}(z)$ and Jack polynomials.''
Adv. Math. {\bf 206} (2006): 307--328.



\bibitem[FKLMM]{FKLMM}
Feigin, B., R. Kedem, S. Loktev,  T. Miwa,  and E. Mukhin.
``Combinatorics of the $\widehat{\mathfrak{sl}}_2$ spaces of coinvariants.''
 Transf. groups {\bf 6}  (2001): 25--52.

\bibitem[FLM]{FLM}
Frenkel, I., and J. Lepowsky, and A. Meurman. ``Vertex operator algebras and the Monster.'' Academic press, 1989.

\bibitem[G]{Geog}
Georgiev, G. ``Combinatorial constructions of modules for infinite-dimensional Lie algebras, I. Principal subspace.'' J.\,Pure Appl.\,Algebra {\bf 112}  (1996) 247--286.

\bibitem[JMM]{JMM}
Jimbo, M., T. Miwa, and E. Mukhin.
``Bosonic formulas for $\widehat{\mathfrak{sl}}_2$ coinvariants.''
Ramanujan J. {\bf 6} (2002): 491--508.


\bibitem[K15]{Kaw15}
Kawasetsu, K. ``The free generalized vertex algebras and generalized principal subspaces.'' J. Alg.  {\bf 444} (2015): 20--51.

\bibitem[K23]{Kaw23a}
Kawasetsu, K. ``On the commutant of the principal subalgebra in the $A_1$ lattice vertex algebra.'' Lett. Math. Phys. {\bf 113} (2023): 123.

\bibitem[KKMM]{KKMM}
Kedem, R., T. Klassen,  B. McCoy, and  E. Melzer. ``Fermionic sum representations for conformal field theory characters.'' Phys. Lett. B {\bf 307} (1993): 68--76.


\bibitem[Ke]{Ke2}
Kenzhaev, T. ``Durfee rectangle identities as character identities for infinite Fibonacci configurations.'' arXiv:2306.08410 (2023).


\bibitem[MP]{MP}
Milas, A., and M. Penn. ``Lattice vertex algebras and combinatorial bases: general case and W-algebras.'' New York J. Math. {\bf 18} (2012): 621--650.

\bibitem[OSS]{OSS}
Okado, M., A. Schilling, and M. Shimozono.
``Crystal bases and $q$-identities.'' $q$-series with applications to combinatorics, number theory, and physics (Urbana, IL, 2000), 29--53.
Contemp. Math., 291
American Mathematical Society, Providence, RI, 2001

\bibitem[PSW]{PSW}
Penn, M., C. Sadowski, and G. Webb. ``Principal subspaces of twisted modules for certain lattice vertex operator algebras.'' Int. J. Math. {\bf 30} (2019): 1950048.

\bibitem[RR]{RR}
Rogers, L., and S. Ramanujan. ``Proof of certain identities in combinatory analysis.'' Proc. Camb. Phil. Soc. {\bf 19} (1919): 211--216.


\bibitem[R]{R}
Roitman, M.
``Combinatorics of free vertex algebras.''
J. Alg. {\bf 255} (2002): 297--323. 

\bibitem[S]{S}
I. Schur, ``Zur Additiven Zahlentheorie.'' Akademie der Wissenschaften, Berlin, Sitzungeberichte (1926): 488--495.

\bibitem[SF]{SF}
Stoyanovskii, A, and B. Feigin. ``Functional models for representations of current algebras and semi-infinite Schubert cells.''
 Funct.\ Anal.\ Appl.\ {\bf 28}  (1994): 55--72.
 

\end{thebibliography}
\end{document}